\documentclass[a4paper,11pt]{amsart}%
\usepackage{hyperref}
\usepackage{pdfsync}
\usepackage{dsfont}
\usepackage{color}
\usepackage{esint}
\usepackage{amsthm}
\usepackage{enumerate}
\usepackage{amsfonts}
\usepackage{amssymb}%
\usepackage{mathtools}
\usepackage{braket}
\usepackage{bm}
\usepackage{amsmath}
\usepackage{graphicx}
\usepackage{caption}
\usepackage{import}
\usepackage{fullpage}
\usepackage{xcolor}
\usepackage{ulem}
\usepackage{mathtools}
\numberwithin{equation}{section}

\usepackage{cite}
\usepackage{float}
\usepackage{subfig}

\usepackage{amsfonts}
\usepackage{graphicx}%
\providecommand{\U}[1]{\protect\rule{.1in}{.1in}}
\newtheorem{theorem}{Theorem}[section]

\newtheorem{corollary}[theorem]{Corollary}

\newtheorem{definition}[theorem]{Definition}
\newtheorem{example}[theorem]{Example}

\newtheorem{lemma}[theorem]{Lemma}

\newtheorem{proposition}[theorem]{Proposition}
\newtheorem{remark}[theorem]{Remark}

\newtheorem{assumption}[theorem]{Assumption}

\DeclareMathOperator*{\argmin}{arg\,min}
\newcommand{\veps}{\varepsilon}
\newcommand{\e}{\varepsilon}

\newcommand{\vol}{\mathrm{vol}}

\newcommand{\T}{\mathcal{T}}

\newcommand{\I}{\mathcal{I}}
\newcommand{\V}{\remmath{\mathfrak{v}}\addmath{\vartheta}}
\newcommand{\capV}{\remmath{\mathfrak{V}}\addmath{\Theta}}
\newcommand{\Cut}{\mathrm{Cut}}
\renewcommand{\r}{\mathfrak{r}}

\newcommand{\R}{\mathbb{R}}

\DeclareMathOperator{\Prob}{\mathbb{P}}
\newcommand{\M}{\mathcal{M}}

\definecolor{mygreen}{rgb}{0.1,0.75,0.2}

\newcommand{\E}{\mathbb{E}}
\renewcommand{\div}{\mathrm{div}}

\newcommand{\eps}{\varepsilon}
\newcommand{\Per}{\mathcal{P}}
\newcommand{\Iso}{\mathbb{I}}

\newcommand{\N}{\mathbb{N}}
 
\newcommand{\vv}{\V}
\newcommand{\CM}{\mathcal{C}_\M}

\newcommand{\Id}{\mathrm{Id}}
\newcommand{\one}{\mathds{1}}
\newcommand{\dd}{\mathrm{d}}
\newcommand{\DD}{\mathrm{D}}
\newcommand{\GTV}{\mathrm{GTV}}
\newcommand{\TV}{\mathrm{TV}}
\newcommand{\BV}{\mathrm{BV}}

\newcommand{\n}{\bm{n}}
\newcommand{\Bal}{\mathrm{Bal}}

\def\l{\left(}
\def\r{\right)}
\def\la{\left|}
\def\ra{\right|}
\def\lda{\left\|}
\def\rda{\right\|}
\def\lb{\left\{}
\def\rb{\right\}}
\def\rd{\right.}

\def\TLp#1{\mathrm{TL}^{#1}}

\def\Wp#1{\mathrm{W}^{#1}}
\def\dWp#1{d_{\Wp{#1}}}
\def\Lp#1{\mathrm{L}^{#1}}
\def\Wkp#1#2{\mathrm{W}^{#1,#2}}

\def\Ck#1{\mathrm{C}^{#1}}

\def\Ckc#1{\Ck{#1}_c}

\definecolor{darkred}{rgb}{0.6,0.1,0.1}
\definecolor{darkgreen}{rgb}{0.1,0.6,0.1}
\definecolor{darkblue}{rgb}{0.1,0.1,0.6}

\def\rem#1{}
\def\add#1{#1}
\def\addmath#1{#1}
\def\remmath#1{}
\def\addii#1{#1}

\providecommand{\keywordsnew}[1]{\textbf{\textit{\noindent  Keywords and phrases. }} #1}
\providecommand{\subjclassnew}[1]{\textbf{\textit{\noindent Mathematics Subject Classification. }} #1}

\title{From graph cuts to isoperimetric inequalities: Convergence rates of Cheeger cuts on data clouds}
\author{Nicol\'as Garc\'ia Trillos}
\address{Department of Statistics, University of Wisconsin, Madison, Wisconsin, USA}
\email{garciatrillo@wisc.edu}
\author{Ryan Murray}
\address{Department of Mathematics, North Carolina State University, Raleigh, NC, USA}
\email{rwmurray@ncsu.edu}
\author{Matthew Thorpe}
\address{Department of Mathematics, University of Manchester, Manchester, Greater Manchester, UK}
\email{matthew.thorpe-2@manchester.ac.uk}

\thanks{NGT was supported by NSF grant DMS 1912802}
\thanks{MT was supported by the European Research Council under the European Union's Horizon 2020 research and innovation programme grant agreement No 777826 (NoMADS) and grant agreement No 647812}

\begin{document}
\normalem

\begin{abstract}
In this work we study statistical properties of graph-based clustering algorithms that rely on the optimization of balanced graph cuts, the main example being the optimization of Cheeger cuts.
We consider proximity graphs built from data sampled from an underlying distribution supported on a generic smooth compact manifold $\M$.
In this setting, we obtain high probability convergence rates for both the Cheeger constant and the associated Cheeger cuts towards their continuum counterparts.
The key technical tools are careful estimates of interpolation operators which lift empirical Cheeger cuts to the continuum, as well as continuum stability estimates for isoperimetric problems.
To our knowledge the quantitative estimates obtained here are the first of their kind.
\end{abstract}

\maketitle

\noindent
\keywordsnew{asymptotic consistency, rates of convergence, isoperimetric inequalities, Cheeger cuts, nonlocal variational problems}

\noindent
\subjclassnew{49Q05, 58E12, 49J55, 49J45,  62G20}


\section{Introduction}

This paper studies statistical properties of graph-based methods in unsupervised learning, and specifically the use of Cheeger cuts for clustering. 
When data points are samples from some underlying distribution on a continuum domain and graphs are proximity graphs (e.g. $\veps$ or $k$-NN graphs), most of the existing literature establishes asymptotic consistency of these methods without providing any rates of convergence.
In this paper, we obtain new estimates for the graph total variation functional on data clouds that, when combined with recent quantitative isoperimetric inequalities on manifolds, allow us, for the first time, to provide high probability convergence rates for both the Cheeger constant (minima) and corresponding Cheeger cuts (minimizers) on graphs built from random data clouds.

The Cheeger cut methodology belongs to a broad family of graph-based methods in data analysis.
These are learning procedures that typically rely on the optimization of objective functions that depend on a similarity graph $G=(\M_n,w)$ on the observed data set $\M_n= \{ x_1, \dots, x_n\}$ and whose role is to capture geometric information from $\M_n$.
Here $w=(w_{ij})_{i,j=1}^n$ represents a weight matrix quantifying the level of similarity between different data points.
Our focus is on unsupervised learning, where the only information on the data available is the graph $G$ itself and no labeling information is present.
In such a scenario, a central goal of learning methods is to use the similarity graph $G$ to construct a meaningful summarized representation of the data set into clusters.
Heuristically, a cluster is a subset of $\M_n$ whose elements are more similar among themselves than when compared with those in other groups. 
This heuristic can be made more precise by describing clusters as groups of points that are well-connected among themselves and less so with elements from other groups.
A concrete mathematical way to use the similarity graph $G$ to capture this intuitive notion is to look for subsets $A$ for which the \emph{cut} functional 
\[ \Cut(A):= \sum_{x_i \in A, x_j \in \M_n \setminus A} w_{ij}, \quad A \subseteq \M_n\]
is small.
As can be seen from the above formula, $\Cut$ penalizes subsets of $\M_n$ that have big interactions with their complements.
An alternative name usually given to this functional is \emph{graph perimeter}, but here we will save this name to refer to an appropriately rescaled version of $\Cut$.

The cut functional has been studied for decades in mathematics and computer science, e.g.~\cite{ARV,bresson2012,Thomas1,kannan04,SzlamBresson,ford56,hagen92,HeinSetz,vonLuxburg07,wei89,spielman04,ShiMalik,arias2012normalized}.
For example, a celebrated result in optimization relates the $\Cut$ minimization problem, subject to hard membership constraints, with a maximum flow problem, a linear program, via duality theory~\cite{ford56}.
In the context of clustering, it is worth noticing that, while direct minimization of $\Cut$ is reasonable as it penalizes the size of the interface between sets, optimization of $\Cut$ on its own is not able to rule out partitions into groups that are highly asymmetric in terms of size, and thus in order to avoid undesirable partitions, the cut functional must be modified appropriately.
For this purpose extra balancing terms are added to the objective function in either additive or multiplicative form so as to penalize ``volume'' asymmetry.
Some prototypical examples of optimization problems in the family of \emph{balanced cuts} are:
\begin{align}
\min_{A\subseteq \M_n} \frac{\Cut(A)}{\min \{\vol_{\M_n}(A), \vol_{\M_n}(A^c) \}}, \label{eq:CheegerCut} \\
\min_{A \subseteq \M_n } \frac{\Cut(A)}{\vol_{\M_n}(A) \cdot \vol_{\M_n}(A^c) } \label{eq:RatioCut}
\end{align}
or 
\begin{equation}
\label{def:ModularityClust}
\min_{A \subseteq \M_n } \Cut(A) + \gamma \left(\vol_{\M_n}(A)^2 + \vol_{\M_n}(A^c)^2 \right).
\end{equation}
The problems~\eqref{eq:CheegerCut} and~\eqref{eq:RatioCut} are known as \emph{Cheeger cut} and \emph{ratio cut} optimization problems respectively (see \cite{Chung1997,vonLuxburg07}).
Problem~\eqref{def:ModularityClust} is known to be equivalent to the problem of \emph{graph modularity clustering} (see~\cite{hu13,Boyd2017SimplifiedEL}); in all the above formulas $\vol_{\M_n}(\cdot)$ stands for the number of elements in a given set divided by $n$.
We observe that, in essence, all the above are $\Cut$ minimization problems (i.e. graph perimeter minimization) with some form of soft volume constraint.
We also remark that in~\eqref{eq:CheegerCut},\eqref{eq:RatioCut} and~\eqref{def:ModularityClust} the target is a two-way partitioning of the data, meaning that the desired partition consists of exactly two sets, namely $A$ and $A^c$, but it is possible to formulate analogous optimization problems in the multi-way clustering setting (e.g.~\cite{hu13,TSvBLX_jmlr,Boyd2017SimplifiedEL}).
In the sequel we focus on the two-way clustering setting.
\vspace{\baselineskip}

Although each of~\eqref{eq:CheegerCut},\eqref{eq:RatioCut} and~\eqref{def:ModularityClust} are quite natural objectives for clustering due to their geometric character, graph partitioning problems are, in their full generality, known to be NP-hard~\cite{kaibel2004expansion}.
Accordingly, in the past two decades a lot of the theoretical effort in the graph-based learning communities focused on analyzing relaxations of these partitioning methods.
For example, in the case of the ratio cut functional~\eqref{eq:RatioCut}, up to a multiplicative constant, the energy associated with $A \subset \M_n$ can be rewritten as
\begin{equation} \label{eq:CFWEigenPair}
\frac{\sum_{i,j}w_{ij}|u(x_i) - u(x_j)|^2}{\sum_{i=1}^n | u(x_i) - \frac{1}{n}\sum_{j=1}^n u(x_j)|^2}
\end{equation}
where $u = \one_{A}$ (assuming weights $w_{ij}$ are symmetric).
One relaxes the problem by allowing $u$ to be an arbitrary function $u: \M_n \rightarrow \R$.
Upon minimizing this relaxed functional, one recovers the variational description for the first non-trivial eigenpair (i.e. the Fiedler eigenpair) of the unnormalized \emph{graph Laplacian}. 
This relaxed problem is convenient from a computational point of view as it only requires an eigensolver for a positive semidefinite matrix, namely the graph Laplacian.
One can subsequently find an appropriate sublevel set via thresholding the Fiedler eigenvector in order to cluster the data into two sets.
This relaxation procedure for constructing graph partitions can be motivated by graph theoretical results such as Cheeger's inequality in its graph version~\cite{mohar89} which, just as in its continuum counterpart due to Cheeger~\cite{cheeger1969lower}, relates the Cheeger constant of the graph (i.e. the minimum value of the Cheeger cut) with the first non-trivial eigenvalue of the graph Laplacian.
In addition, higher-order eigenmodes of the graph Laplacian are useful for multi-way clustering: this idea lies at the heart of the celebrated \emph{spectral clustering} algorithm~\cite{ng2002spectral}.
The spectrum of the graph Laplacian can also be used to extract multiscale geometric information from the original data set $\M_n$ by means of appropriate \emph{diffusion maps}~\cite{coifman2006diffusion,Coifman7426,Portegies}.

The general utility of these Laplacian-based methods can be seen by the wide body of literature analyzing their statistical properties in supervised, semi-supervised and unsupervised settings.
Many of these works aim to provide specific connections between these graph-based methodologies with analogous geometric and analytical notions on manifolds and other objects at the continuum level; this goal is summarized with the term \emph{manifold learning}.
A typical model assumption of analysis in this field is to think of points $\M_n=\{ x_1, \dots, x_n \}$ as samples from some distribution supported on a low dimensional manifold $\M$ embedded in an ambient space $\R^d$, and to consider weights $w$ that are determined by the proximity of points in $\R^d$. Namely, high weights are given to pairs of points that are close to each other.
A popular construction involves the choice of a connectivity parameter $\veps>0$, defining weights by $w_{ij}= \eta\left( \frac{|x_i-x_j|}{\veps} \right)$, where $\eta$ is a non-increasing function decaying fast enough to zero.
Over the past two decades many theoretical results have been obtained for the asymptotic consistency of graph Laplacians towards continuum limits.
These include pointwise consistency~\cite{HeinvonLuxburgAudibert,Singer,CalderpLap,calder20}, and stronger spectral consistency~\cite{BuragoIvanovKurylev,CalderTrillos,vonLuxburg,belkin2007convergence,GTSSpectralClustering,garciatrillos19} which were more suitable for the proposed learning methods.
Subsequently, these analytical tools have facilitated a deeper understanding of various statistical methods \cite{dunlop2019large,CalderSlepcev,NGT-RM,el2016asymptotic,TASK}.
\vspace{\baselineskip}

In addition to this theoretical development of graph Laplacian methods, in the past decade there has been a renewed interest among researchers in analyzing the original un-relaxed cut functional and its use for data clustering.
On the one hand, this was driven by significant theoretical activity in the image processing communities, where total variational functionals provide natural energies for describing sharp edges (see \cite{ChambolleNovaga} for a broad overview).
On the other hand, new algorithmic improvements for total variation minimization made the original family of balanced cut optimization problems more accessible~\cite{chambolle11,Thomas1,bluv13,HeinSetz,SzlamBresson,HeinBuhl}.
This renewed interest also motivated further theoretical analysis from a statistical point of view.
Analyzing balanced cut type problems, such as~\eqref{eq:CheegerCut} and~\eqref{eq:RatioCut}, is difficult because the associated optimization problems are highly non-convex, with an objective function that depends non-trivially on random data.
One approach to analyzing large sample properties of the Cheeger cut problem is given in ~\cite{narayanan09,arias2012normalized}.
These works consider the minimization of the Cheeger cut functional on a $\veps$-graph, but only over subsets of $\M_n$ that are restrictions to the data set of sufficiently regular subsets of the underlying manifold $\M$.
In this constrained setting, the classical approaches of statistical learning that rely on the computation of VC dimensions or other capacity measures were used to deduce convergence rates of the ``restricted" discrete Cheeger constant towards a corresponding counterpart at the manifold level.
To our knowledge these are the only works which obtain convergence rates for Cheeger constants based upon continuum sampling.
However, in many practical settings the manifold $\M$ is unknown, and hence the constraints imposed in \cite{arias2012normalized} may be difficult to guarantee in applications.

Variational approaches to analyzing cuts on graphs in the large data limit were, to our knowledge, first introduced in~\cite{Bertozzi}.
There, the authors established a link between discrete and continuum Cheeger cuts via $\Gamma$-convergence type arguments for a Ginzburg--Landau approximation of the cut functional defined on a graph.
The analysis, however, required the data points to form a regular lattice, and in particular, it was not clear how to extend the ideas to data sampled from a distribution or more irregular point arrays. 
The work~\cite{GarciaTrillos2015} introduced a framework, built upon both optimal transportation and variational methods, which rigorously connects the Cut functional on graphs (built from random data) and continuum total variation.
These tools provide an almost optimal condition on the scaling of $\veps$ (i.e. the connectivity length scale of the proximity graph) in terms of $n$ for the consistency to hold.
In order to link the discrete and continuum energies and their minimizers, the authors introduced the so-called $\TLp{p}$ distance.
The key idea with these spaces is to consider couplings between measures and functions, i.e. $(\mu,f)\in \TLp{p}$ if $\mu$ is a probability measure and $f\in\Lp{p}(\mu)$.
A metric is then defined between couplings $(\mu,f)$, $(\nu,g)$ that is equivalent to the Wasserstein distance between $(\Id\times f)_{\#}\mu$ and $(\Id\times g)_{\#}\nu$.
With this topology in hand the authors provide a link between graph and continuum total variations via $\Gamma$-convergence (e.g.~\cite{DalMaso,braides2002gamma}), a tool developed precisely in order to describe the convergence of variational problems.
This analytical framework has been applied directly to the Cheeger (and other) cut problems in~\cite{trillos2017estimating,TSvBLX_jmlr}, but without rates. 
It has also been applied to the Ginzburg--Landau approximation to extend the results in~\cite{Bertozzi} to the random data setting in~\cite{thorpe19,cristoferi18AAA}. 
In order to avoid non-degeneracy of a graph total variation type functional, the connectivity radius $\eps$ cannot scale to zero too quickly. 
Originally this required the connectivity radius $\eps$ to be much greater than $\infty$-Wasserstein distance between the empirical measure $\nu_n=\frac{1}{n}\sum_{i=1}^n \delta_{x_i}$ and the data generating measure $\nu$ (i.e. $\eps\gg\dWp{\infty}(\nu_n,\nu)$).
In dimensions $d\geq 3$ the connectivity of the graph is of the same order as $\dWp{\infty}(\nu_n,\nu)$, and for $d=2$ there is a logarithmic correction factor which meant a small gap between the two~\cite{PenroseBook,W8L8canadian}.
In~\cite{muller18AAA} the exponent for the two dimensional case was improved, and as of right now, consistency of Cheeger constants and cuts is known to hold provided the graph connectivity is asymptotically above the connectivity threshold.
All of this without providing any quantitative rates of convergence.

In summary, the works that have studied the convergence of graph Cheeger constants and corresponding Cheeger cuts towards continuum counterparts have done so without providing high probability convergence rates, while the only works that provide some convergence rates for Cheeger constants (and not for their associated optimal cuts) do so by modifying the original problem in a manner that is not fully satisfactory from the point of view of applications.
Thus, the question of finding high probability convergence rates for graph Cheeger constants and cuts on proximity graphs, to the authors' knowledge, has not been addressed anywhere in the literature.
Our goal is to provide a first work in this direction.
\vspace{\baselineskip}

In this paper we present high probability rates for the convergence of both Cheeger constants (i.e. minima) and corresponding Cheeger cuts (i.e. minimizers) towards their analogous continuum counterparts.
Our approach will strongly highlight how analytical ideas can help provide answers to problems that have been elusive using traditional tools from statistics or statistical learning theory.
Our contribution will begin by establishing new estimates on the graph total variation functional on proximity graphs.
These estimates will allow us to obtain high probability convergence rates for Cheeger constants.
Our second contribution will be to connect the problem of convergence of minimizers (i.e. cuts) to recent results in geometric measure theory related to quantitative isoperimetric inequalities.
Quantitative isoperimetric inequalities are a family of relations that hold at a \emph{continuum} level (i.e. in $\R^d$ or on generic manifolds).
Heuristically, they capture a type of ``strong local convexity'' of the volume constrained perimeter minimization problem (at the continuum level) near minimizers.
One of the morals we hope to convey in this paper is that in the geometric setting, where continuum variational problem often possesses significant structure, it is often possible to deduce strong statistical properties of data analysis procedures, which are defined at the \emph{discrete}, finite sample level.
We firmly believe that the bridge between data analysis and mathematical analysis is both necessary and promising.

Finally, it is worth saying that we do not claim optimality of our convergence rates, and in fact we believe that further improvement to our results is possible.

\subsection{Outline}
The rest of the paper is organized as follows.
In Section~\ref{sec:problemsetup} we present our problem setup and state our main results precisely.
In particular, Sections~\ref{sec:setupDiscrete} and~\ref{sec:contFramework} describe the discrete and continuum problems that we aim to connect in this paper.
In Section~\ref{sec:CheegerConstants} we present our first main result (Theorem~\ref{thm:cheeger-constant-conv}) stating the high probability convergence rates of Cheeger constants (i.e. convergence of minima).
In Section~\ref{subsubsec:problemsetup:CheegerCutsDetour} we take a small detour and present a short discussion on isoperimetric problems, and discuss the connection of these problems with the problem of establishing convergence rates for Cheeger cuts (i.e. convergence of minimizers).
Section~\ref{subsubsec:problemsetup:CheegerCutsRes} contains our convergence rates for Cheeger cuts (Theorem~\ref{thm:rates-Cheeger-Cuts}).
Section~\ref{subsec:problemsetup:IsoStabRev} expands the discussion on isoperimetric stability and provides pointers to the literature.

Section~\ref{sec:nonlocal} contains a series of analytical results on non-local total variation energies.
We wrap up the paper in Section~\ref{sec:ProofMainResults} where we present the proofs of our main results.

\section{Problem Setup and Main Results}
\label{sec:problemsetup}

Throughout the paper $\M$ will denote a smooth, connected, orientable, compact, Riemannian manifold of dimension $m$, without boundary, embedded in $\R^d$. 
In particular, the Riemannian metric tensor on $\M$ is the one inherited from $\R^d$.
This will allow us to relate the geodesic distance on $\M$ with the Euclidean metric for points on $\M$ that are close enough.
We use $\vol_\M$ to represent $\M$'s volume form, and without \rem{the} loss of generality \rem{we we will} assume that the manifold $\M$ is normalized so that $\vol_\M(\M)=1$, and in particular $\nu:=\vol_{\M}$ is a probability measure (the uniform distribution) on $\M$.
We let $i_\M>0$ be the injectivity radius of $\M$, which defines the maximum radius of a $m$-dimensional ball centered at the origin of an arbitrary tangent plane $\T_x\M$ for which the exponential map $\exp_x: \T_x\M \rightarrow \M$ is a diffeomorphism.
Other notions from differential geometry that are used in the paper are discussed in Chapters 1-3 in~\cite{docarmo1992riemannian}, and are introduced as needed. 

Throughout the paper we will use letters like $v , \tilde v$ to represent tangent vectors on $\M$.
We will use $u,\tilde u$ to represent discrete functions from $\M_n$ into $\R$, and use $f,\tilde f$ to represent functions from $\M$ into $\R$. 
For two sets we write $A \Delta B = (A \backslash B) \cup (B \backslash A)$ to denote their symmetric difference.

\subsection{Discrete Set-Up}
\label{sec:setupDiscrete}

We will use $\M_n=\{x_i\}_{i=1}^n$ to represent the discrete approximation of the manifold $\M$ via sampling from a distribution supported on $\M$.
We use $\nu_n = \frac{1}{n}\sum_{i=1}^n \delta_{x_i}$ to represent the associated empirical measure.
The graph $G_n$ is defined to be the set of vertices $\M_n$ with edge weights $\{w_{ij}\}_{i,j=1}^n$, representing the similarities between nodes $x_i$ and $x_j$.
Here we consider weights of the form
\[ w_{ij} := \eta \left( \frac{|x_i-x_j|}{\veps} \right),\]
where $\e>0$ is a small parameter representing a length scale on which we consider points to be neighbors.
We will assume for simplicity that the kernel $\eta: \mathbb{R}^+ \to \mathbb{R}^+$ takes the form
\[ \eta(t) = \one_{\{t\leq 1\}} = \lb \begin{array}{ll} 1 & \text{if } t\leq 1 \\ 0 & \text{else} \end{array} \rd \]
although the analysis presented here can be extended to other choices of decreasing kernel with only minor modifications needed.
We will, later on, rescale these edge weights in order to obtain appropriate asymptotic limits.

We will assume for simplicity that the points $\{x_i\}_{i=1}^n$ are sampled from the distribution $\nu$ with density given by $\rho \equiv 1$ (with respect to the natural volume form $\vol_\M$).
The extension to smooth, non-degenerate $\rho$ does not cause any significant technical changes, but burdens the notation: we only consider the uniform case for clarity.

Given two arbitrary functions $u, \tilde u : \M_n \rightarrow \R $ we define their inner product:
\[ \langle u, \tilde u \rangle_{\Lp{2}(\M_n)}:= \frac{1}{n}\sum_{i=1}^n u(x_i) \tilde u(x_i), \] 
as well as the $\Lp{2}(\M_n)$ norm
\[ \lVert u \rVert_{\Lp{2}(\M_n)} := \sqrt{\langle u, u\rangle_{\Lp{2}(\M_n)}} \]
and the $\Lp{1}(\M_n)$ norm
\[ \lVert u \rVert_{\Lp{1}(\M_n)} := \frac{1}{n}\sum_{i=1}^n |u(x_i)|. \]

For a function $u: \M_n \to \R$, we also define the \emph{graph total variation seminorm} as: 
\begin{equation}
\label{def:GTV}
\GTV_{n,\veps}(u):= \frac{1}{n^2\veps^{m+1}} \sum_{i=1}^n \sum_{j=1}^n w_{ij}| u(x_i) - u(x_j)|.
\end{equation}
We notice that when $u$ is an indicator function of a subset $E_n$ of $\M_n$ we get the following:
\[ \GTV_{n,\veps}(\one_{E_n}) = \frac{2}{n^2\eps^{m+1}} \Cut(E_n). \]

In the remainder we will use $\mathcal{C}_{n,\veps}$ to denote the rescaled graph Cheeger constant
\begin{equation}
\label{eqn:CheegerGraph} 
\mathcal{C}_{n,\veps}:= \min_{E_n \subseteq \M_n} \frac{\GTV_{n,\veps}(\one_{E_n})}{ \min \{ \nu_n(E_n), 1 - \nu_n(E_n) \} },
\end{equation}
and use $A_n^*$ or $E_n^*$ to denote an arbitrary minimizer.

\subsection{Continuum Framework}
\label{sec:contFramework}

We consider the classical spaces $\Lp{1}(\M)$ and $\Lp{2}(\M)$ with norms
\[ \lVert f \rVert_{\Lp{1}(\M)}:= \int_{\M} |f(x) | \, \dd \vol_\M(x), \qquad \lVert f \rVert_{\Lp{2}(\M)}:= \left(\int_{\M} |f(x) |^2 \, \dd \vol_\M(x)\right)^{1/2}. \]

For a given $f \in \Lp{1}(\M)$ we define its \emph{total variation} seminorm as
\begin{equation}
\label{def:TV}
\TV(f):= \sup \left\{ \langle \div(V) , f \rangle_{\Lp{2}(\M)} \: : \: V \in \mathfrak{X}(\M), \quad |V(x)|_x \leq 1, \quad \forall x \in \M \right\},
\end{equation}
and say that $f \in \BV(\M)$ if $\TV(f) < \infty$.
In the above, $\mathfrak{X}(\M)$ denotes the set of all smooth vector fields on $\M$, $\div$ is the divergence operator (on $\M$) mapping smooth vector fields into real valued smooth functions on $\M$, and $|\cdot|_x$ is the Riemannian norm in the tangent space at $x\in\M$.
We define the \emph{perimeter} of a measurable $E \subset \M$ via
\[ \Per(E) := \TV(\one_E). \]
We notice that the expression~\eqref{def:TV} reduces to
\begin{equation}
\TV(f)= \int_{\M} |\nabla f (x)|_x \, \dd\vol_\M(x) ,
\label{eqn:smoothf}
\end{equation} 
when $f$ is a smooth function (i.e. $f \in \Ck{\infty}(\M)$), where here $\nabla f$ is the gradient (with respect to $\M$'s Riemannian metric) of $f$.
Likewise, the perimeter of a subset $E \subseteq \M$ with smooth boundary $\partial E$ can be written as 
\[ \Per(E) = \int_{\partial E} \dd \mathcal{H}^{m-1}(x), \]
where $\mathcal{H}^{m-1}$ is the $m-1$ dimensional Hausdorff measure on $\M$. 

We recall that the \emph{coarea} formula allows us to write the total variation of a $\BV$ function in terms of the perimeter of its level sets. 
Namely, for a given $f \in \BV(\M)$ we have
\begin{equation}
\label{eqn:Coarea}
\TV(f) = \int_{-\infty}^{\infty} \mathcal{P}(\{ x \in \M \: : \: f(x) \leq t \}\} ) \, \dd t. 
\end{equation}

In the remainder we will use $\CM$ to denote the Cheeger constant
\begin{equation} 
\label{eqn:CheegerCont}
\CM:= \min_{E \subseteq \M} \frac{\Per(E)}{ \min \{ \nu(E), 1 - \nu(E) \} },
\end{equation}
and use $A^*$ or $E^*$  to denote an arbitrary minimizer.

\subsection{Main Results}
\label{sec:mainresults}

For ease of reference we summarize the assumptions on the manifold $\M$ and the length scale $\eps$.

\begin{assumption}
\label{assumptions}
Let $\M$ be a smooth, connected, orientable, compact, Riemannian manifold of dimension $m$, without boundary, embedded in $\R^d$.
We assume that $\eps\leq \eps_0$ where $\eps_0>0$ is some sufficiently small parameter and, in particular, is smaller than the injectivity radius $i_\M$.
%
%
%
%
\end{assumption}

We have two main results: the first is a rate of convergence for the Cheeger constants (Theorem~\ref{thm:cheeger-constant-conv}) and the second is a rate of convergence for Cheeger cuts (Theorem~\ref{thm:rates-Cheeger-Cuts}).
We give the result for the convergence of Cheeger constants in the next section.
The convergence of Cheeger cuts requires more notation and assumptions than we have introduced thus far so we include in Section~\ref{subsubsec:problemsetup:CheegerCutsDetour} a detour into isoperimetric stability before stating our rate of convergence for Cheeger cuts in Section~\ref{subsubsec:problemsetup:CheegerCutsRes}.

\subsubsection{Convergence Rates for Cheeger Constants}
\label{sec:CheegerConstants}

Our first main result establishes a quantitative convergence rate of the graph Cheeger constant~\eqref{eqn:CheegerGraph} towards the Cheeger constant on the manifold $\M$~\eqref{eqn:CheegerCont}.
The precise theorem is as follows.

\begin{theorem}
\label{thm:cheeger-constant-conv}
(Cheeger constants)
\rem{Under assumptions \ref{assumptions} on $\veps$ and $\M$, and assuming $\delta, \theta, \zeta>0$ are all small enough quantities and in particular $\delta \leq \veps/4$, 
there exist constants (that may depend on $\M$) $C_1,C_2,C,c,c'>0$ such that if $n\zeta\eps^{\frac{m+1}{2}}\geq c$ then, with probability at least $1- n \exp(-cn\theta^2\delta^m)-C \exp\left(-cn\zeta \min\{\e^{\frac{m+1}{2}},\eps\zeta\}\right)$, we have:}
\add{Let $\M$ and $\veps$ satisfy Assumptions~\ref{assumptions}.
Then, there exists constants (that may depend on $\M$) $\theta_0,\zeta_0,C_1,C_2,C,c,c'>0$ such that for any $\delta,\theta,\zeta>0$, with $n\zeta\eps^{\frac{m+1}{2}}\geq c$, $\delta\leq \frac{\veps}{4}$ , $ c' \log(n)/n \leq \theta^2 \delta^m $, $\theta\leq \theta_0$ and $\zeta\leq \zeta_0$, we have} \addii{that}\add{, with probability at least $1- n \exp(-cn\theta^2\delta^m)-C \exp\left(-cn\zeta \min\{\e^{\frac{m+1}{2}},\eps\zeta\}\right)$,}
\begin{enumerate}
\item Upper bound:
\[ \mathcal{C}_{n,\veps} \leq \sigma_\eta \CM + C_1\left(\e^2 + \zeta\right). \]
\item Lower bound:
\[ \sigma_\eta \CM \leq \mathcal{C}_{n,\veps} + C_2 \left( \sqrt[3]{\eps} + \frac{\delta}{\veps} + \theta + \zeta \right), \]
\end{enumerate}
where $\sigma_\eta$ is the ``surface tension''
\begin{equation}
\label{def:sigmaeta}
\sigma_\eta:= \int_{\R^m}\eta(|z|)|z_1|\, \dd z = \int_{B(0,1)} |z_1| \, \dd z,
\end{equation}
and in the above, $z_1$ stands for the first coordinate of a given vector $z \in \R^m$.
\end{theorem}

The parameter $\eps$ has already been defined as the length scale determining when two data points should be considered neighbors.
The parameters $\delta$ and $\theta$ control how fast the empirical measure $\nu_n$ is converging.
More precisely we consider a smooth approximation $\widetilde{\nu}_n\in\mathcal{P}(\M)$ of $\nu$ with the properties that the density $\widetilde{\rho}_n$ of $\widetilde{\nu}_n$ satisfies $\widetilde{\rho}_n(x)\in[1-\delta-\theta,1+\delta+\theta]$ and $\dWp{\infty}(\nu_n,\widetilde{\nu}_n)\leq \delta$ for every $x\in\M$.
Such a $\widetilde{\nu}_n$ exists with probability at least $1-ne^{-cn\delta^m\theta^2}$.
The reason for using the intermediary probability measure $\widetilde{\nu}_n$ is to get a better probability bound when $m=2$; in particular, for $m>2$ we could use the bound $\dWp{\infty}(\nu_n,\nu)\leq \delta$ with probability at least $1-Cn^{-\alpha}$ for any $\alpha>1$~\cite{garciatrillos19}.
When $m=2$ there is an additional logarithmic correction factor which makes the rates suboptimal (and rather cumbersome to continually give $m=2$ a different rate).
However, the introduction of the intermediary $\widetilde{\nu}_n$ avoids these problems and we recover optimal (up to constants) probability bounds.

The final parameter $\zeta$ comes from a concentration inequality for U-statistics.
Our result is approximately that for each function $u:\M\to\R$ the difference between $\GTV_{n,\eps}(u)$ and $\mathbb{E}[\GTV_{n,\eps}(u)]$ can be bounded by $\zeta$ with high probability (where the ``high probability'' depends on $\zeta$).

\begin{remark}\label{rem:Cheeger-rates}
Of course, one can choose the various parameters $\delta,\theta,\zeta$ in order to obtain concrete convergence rates with high probability as a function of the number of data points $n$.
In particular, if we neglect log terms, assume $m\geq 4$, and set the parameters as follows (by optimizing first the lower bound, and then subsequently the upper bound): 
\begin{align*}
\delta &= n^{-k_\delta},\qquad & k_\delta & = \frac{2}{1+2m} \\
\theta &= n^{-k_\theta},\qquad & k_\theta & = \frac{1}{2(1+2m)} \\
\e &= n^{-k_\e},\qquad & k_\e & = \frac{3}{2(1+2m)} \\
\zeta &= n^{-k_\zeta},\qquad & k_\zeta & = \frac{3}{1+2m}. \\
\end{align*}
This gives, for the lower bound, a rate of $n^{-\frac{1}{2+4m}}$ and for the upper bound a rate of $n^{-\frac{6}{2+4m}}$. 
We notice that in these rates $\delta \ll \e$ and $n\zeta\eps^{\frac{m+1}{2}}\gg 1$ as required, and that the argument of the exponentials in the probability estimates are all bounded away from zero (and hence could be made large by using appropriate logarithmic corrections).
Hence, after neglecting these logarithmic terms, we obtain converge rates of the Cheeger constants of order $n^{-\frac{1}{2+4m}}$. 
We also notice that as long as $\left( \frac{\log(n)}{n}\right)^{1/m}\ll \veps$, then the convergence of $\mathcal{C}_{n,\veps}$ towards $\sigma_\eta C_\M$ as $\veps\rightarrow 0$ and $n \rightarrow \infty$ is guaranteed.
\end{remark}

\begin{remark}
\label{rem:CheegerConsRateOpt}
We recall that the optimization problem in~\eqref{eq:CFWEigenPair} is a relaxation of the closely related ratio cut problem~\eqref{eq:RatioCut}.
In particular, the variational problems coincide when one restricts the minimization in~\eqref{eq:CFWEigenPair} to be over functions of the form $u=\one_A$.
By the Courant-Fischer-Weyl min-max principle the minimization of~\eqref{eq:CFWEigenPair} over $u:\M_n\to\R$ is equivalent to finding the Fiedler eigenpair of the graph Laplacian; in particular the minimum is the first non-zero eigenvalue and the minimizer the corresponding eigenvector.
In~\cite{garciatrillos19} the authors use similar techniques to show that the eigenvalue converges with rate (ignoring logarithms) $n^{-\frac{1}{2m}}$, which is better by approximately a squared factor (for large $m$ and ignoring logarithms the rate for the Cheeger constant is $n^{-\frac{1}{4m}}$).
To go between discrete and continuum both this paper and~\cite{garciatrillos19} use a smooth interpolating operator; in particular, a piecewise constant interpolation followed by mollification. 
The reason for the difference in rates is that in\cite{garciatrillos19} the authors can choose the mollifying kernel based on the choice of interaction potential $\eta$ that allowed a tighter control over the Dirichlet energy of the continuum approximation.
Here, there is not a natural choice of mollifying kernel and hence we must introduce an extra length scale at which we can control the total variation of the continuum approximation.
We also point out that the scaling in~\cite{garciatrillos19} is also suboptimal and has recently been improved to $n^{-\frac{1}{m+4}}$ using different techniques~\cite{CalderTrillos}.
\end{remark}

As we will show in Section \ref{sec:Upper} the upper bound in Theorem \ref{thm:cheeger-constant-conv} is obtained by comparing the graph total variation $\GTV_{n,\veps}$ and the total variation $\TV$ of a \textit{fixed} BV function $f : \M \rightarrow \R$, namely, a minimizer $f=\mathds{1}_{E^*}$ for the continuum Cheeger problem.
The probabilistic estimates behind this comparison rely on general concentration inequalities for $U$-statistics from \cite{GineUStats}.
It will become clear from our analysis that the error estimates for the upper bound are much tighter than the ones for the lower bound, where a simple pointwise convergence estimate (i.e. fixing a function on $\M$) does not suffice.
To be able to prove the lower bound, one of the main technical tools that we introduce is the construction of an interpolation map that relates subsets of $\M_n$ with subsets of $\M$, in a way that is possible to keep track of the error of approximation in a quantitative form.
This map is introduced in Section \ref{sec:Interpolation}.
Section \ref{sec:nonlocal} will lay the groundwork for this construction, and in particular we will prove several results relating the total variation functional $\TV$ and the \textit{non-local} TV seminorms $\TV_h$ defined by
\begin{equation}
\label{eqn:NonLocalTvseminorm}
\TV_{h}(f):= \frac{1}{h^{m+1}}\int_\M \int_\M |f(x) - f(y)| \eta\left( \frac{d_\M(x,y)}{h} \right)\dd \vol_\M(x) \dd \vol_\M(y), \quad f \in \Lp{1}(\M), 
\end{equation}
where here and in the remainder $d_\M(x,y)$ denotes the geodesic distance between the points $x,y \in \M$.
It is worth remarking that there are some connections between the interpolation map that we define here and the one used in \cite{BuragoIvanovKurylev} when analyzing \textit{Dirichlet} energies.
We will discuss more on this in Section \ref{sec:nonlocal}.

\subsubsection{Convergence Rates for Cheeger Cuts: a Detour into Isoperimetric Stability}
\label{subsubsec:problemsetup:CheegerCutsDetour}

A significant part of this work is devoted to proving convergence rates of the minimal \emph{values} of graph based TV energies towards the minimal \emph{values} of TV energies.
If our objective functions were a-priori strongly convex, one could immediately infer quantitative convergence rates for minimizers.
For example, in the case of the Rudin-Osher-Fatemi~\cite{rudin1992nonlinear} type problems, one seeks to minimize a functional given by $J(u) = \TV(u) + \|u-u_0\|_{\Lp{2}}^2$; here the strong convexity of the $\Lp{2}$ penalty could then be used to infer rates of convergence in $\Lp{2}$ using only information about the value of $J$.
The same principle holds in the case of trend filtering models; arguments regarding how our tools apply in those settings to those problems will be given in future works.

However, for Cheeger cut problems it is not obvious in what sense the TV semi-norm is strongly convex (at least locally around minimizers).
Fortunately, a significant body of recent work studying \emph{isoperimetric stability} provides tools for addressing this issue.
We will begin by describing the type of results from this work, and then give an overview of developments in this field.
To begin, we must first give some geometric definitions.

\begin{definition}
We define the {\bf isoperimetric \rem{function} \add{profile}} of our manifold $\M$ via
\[ \Iso(\vv) = \inf_{E \subseteq \M}\lb \Per(E) : \vol_\M(E) = \vv\rb. \]
We define the {\bf Fraenkel asymmetry} of a set via
\[ \alpha(E) = \inf_{E^*}\lb\vol_\M(E \Delta E^*) : \vol_\M(E^*)= \vol_\M(E), \Per(E^*) = \Iso(\vol_\M(E^*))\rb. \]
\end{definition}

In the case of continuum Cheeger problem, any minimizer is a mass-constrained perimeter minimizer, and hence the Cheeger problem reduces to minimizing $\frac{\Iso(\vartheta)}{\min(\vartheta,1-\vartheta)}$ \add{over $\vartheta \in (0,1)$}. The Fraenkel asymmetry measures the distance between a given set and the family of mass constrained perimeter minimizers.
This turns out to be the natural $\Lp{1}$ type distance for this problem, because in many manifolds with symmetries there exists a continuous family of mass-constrained perimeter minimizers.
One example is translates of a ball in $\R^d$. An analogous example on a compact manifold is given in Example~\ref{ex:torus}.

Having defined the energy and distance of interest, we now define variations of the boundary of a set (an in-depth description of these objects can be found in, e.g.,~\cite{Maggi-Book}).

\begin{definition}
Given $E \subset \M$ with smooth boundary $\partial E$, we let $\mathcal{D}_{E}$ be the family of all smooth $\phi: \M\times[-1,1] \to \M$ so that $\phi(x,0) = x$ and $\vol_\M(\phi(E,t)) = \vol_\M(E)$. 
We let $v_\phi = \frac{\dd}{\dd t} \phi(\cdot,0)$ be the initial velocity of the diffeomorphism $\phi$.
This is the family of smooth diffeomorphisms which preserve the mass of $E$.

We let
\[ \mathcal{B}_\delta(\partial E) = \lb f \in \Ck{2,\alpha}(\partial E;\R) : |f|_{\Ck{2,\alpha}} < \delta,\quad \int_{\partial E} f \, \dd \sigma_{g,E} = 0\rb, \]
where $\dd\sigma_{g,E}$ is the volume form that $\partial E$ inherits from $\M$.

We let $\n_E$ be a vector field which restricted to $\partial E$ is an outward unit normal vector field, and which is (in a $\delta$ tube around $\partial E$) a geodesic vector field.
For $f \in \mathcal{B}_\delta(\partial E)$, we let $E + f = \phi(E,1)$ for $\frac{\dd}{\dd t} \phi = \n_{E} \tilde f$, where $\tilde f$ is a smooth extension of $f$.
In other words, $E + f$ is the set obtained by flowing the boundary points of $E$ a distance $f$ along geodesics which are normal to $E$.

We write $D\Per(E)$ and $D^2 \Per(E)$ to represent the first and second variations of the functional $\Per(E + f)$ for $f \in \mathcal{B}_\delta(\partial E)$.
\end{definition}

We remark that the previous definitions have used the assumed regularity on $\partial E$, in particular when choosing $\delta$ so that $\n_{E}$ will be a geodesic field.
The regularity of mass constrained perimeter minimizers is a well-studied topic. In fact, the boundary of such a minimizer is known to be analytic at any point in the reduced boundary (namely where a weak normal vector can be defined); this is related to the fact that the boundary will have constant mean curvature, which can be converted into an equation of elliptic type. Furthermore, the set of points which are not in the reduced boundary for such minimizers, henceforth called \emph{singular points}, is known to have dimension at most $m-8$. Details of this theory may be found in ~\cite{Maggi-Book}, and further discussion is given later in this section.

Having stated this, we follow~\cite{chodosh2019riemannian} and assume that $\M$ is such that the boundary of a mass constrained perimeter minimizer $E^*$ has no singular points, meaning that it is $\Ck{\infty}$ (which only requires us to assume that it is $\Ck{3}$); see Assumptions \ref{assumption:Cheeger-sets} below. We also notice that we have restricted to mean zero variations in the definition of  $\mathcal{B}_\delta(\partial E^*)$ in order to preserve the volume of regions along variations.

In any case, the previous definition describes variations of a set in terms of deformations by smooth vector fields.
Using these definitions of variations, one can easily establish necessary conditions for a set to be a local minimizer of the perimeter.
In particular, we will be interested in the following necessary second-order conditions for local perimeter minimizers.

\begin{definition}
A smooth minimizer $E^*$ of the mass constrained isoperimetric problem is called {\bf strict} if there exists $c>0$ such that for any $f \in \mathcal{B}_\delta(\partial E^*)$
\[ D^2 \Per(E^*) [f,f] \geq c\|f \|_{\Wkp{1}{2}(\partial E^*)}^2. \]

A mass constrained perimeter minimizer $E^*$ is called {\bf integrable} if for any $f \in \mathcal{B}_\delta(\partial E^*)$ with $D^2 \Per(E^*)[f,f] = 0$ there exists a diffeomorphism $\phi$ satisfying \add{$\frac{\dd}{\dd t} \phi (\cdot,0) = f$} and $\phi(x,0) = x$ which satisfies $\vol_\M(\phi(E^*,t)) = \vol_\M(E^*)$ such that $\phi(E^*,t)$ is a critical point of the \add{mass-constrained perimeter} for all $t \in (-1,1)$.
\end{definition}

The definition of strict minimality can be seen as providing a type of Poincar\'e inequality for variations of the boundary of $E^*$; the importance of such Poincar\'e inequalities is further demonstrated in Example~\ref{ex:1D}.
Indeed, strict minimality is stronger than being an isolated local minimizer.

For the sake of illustration, we now give two examples of perimeter minimizers which are, respectively, strict and integrable, followed by an example where the assumption fails.

\begin{example}
Suppose that $\M$ is the ellipsoid shown in Figure~\ref{fig:Ellipse}.
Then the Cheeger set will be a strict mass-constrained perimeter minimizer for mass equal to $\vol_\M(\M)/2$.
\end{example}

\begin{example}
\label{ex:torus}
Suppose that $\M$ is a two dimensional torus embedded in $\R^3$, as shown in Figure~\ref{fig:Torus}.
The Cheeger set shown is a mass constrained perimeter minimizer for mass equal to $v=\vol_\M(\M)/2$.
In this case this minimizer is not strict, but will be integral, as there is a family of mass constrained perimeter minimizers given by rotating the set around the torus.
\end{example}

\add{We note that the one-parameter family of mass-constrained perimeter minimizers in this case is associated with an underlying symmetry of the manifold. In the Euclidean case (i.e. $\M = \R^d$), the translational symmetry always has to be accounted for in isoperimetric problems, and indeed provides significant motive to study integrable minimizers.}

\begin{example}\label{ex:non-strict}
Suppose that we modify the previous torus example so that the radius of the circles varies $\pi$ periodically in $\theta$ according to some function $f(\theta)$.
The Cheeger set, which corresponds to the global mass-constrained perimeter minimizer for volume equal to $\vol_\M(\M)/2$, should partition the manifold at local minimizers of $f$.
However, if the second derivative of $f$ is zero at that local minimizer, then we should not have that such a mass-constrained perimeter minimizer is strict in the sense given above.
This example relies on the $\pi$-symmetry of $f$ in order to handle the mass-constrained variations; a degenerate well is not sufficient by itself.
\end{example}

\add{Clearly not all manifolds have perimeter minimizers which are either strict or integrable, but the situation in Example \ref{ex:non-strict} where the assumption fails seems somewhat degenerate. It is possible that many manifolds of interest will have Cheeger sets which satisfy these properties.}

\begin{figure}
\centering
\subfloat[][]{\includegraphics[width=.5\textwidth]{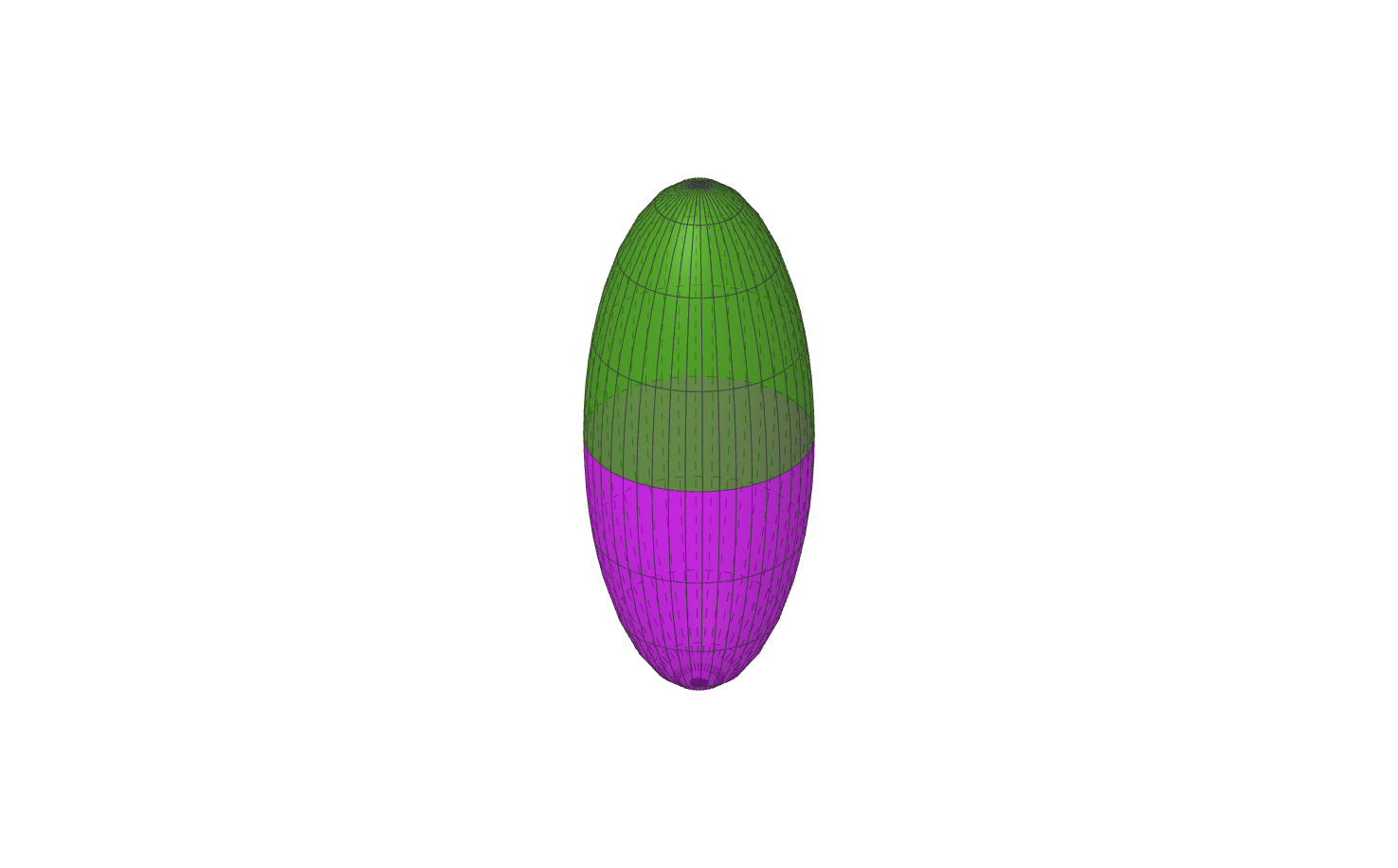}\label{fig:Ellipse}}
\subfloat[][]{\includegraphics[width=.5\textwidth]{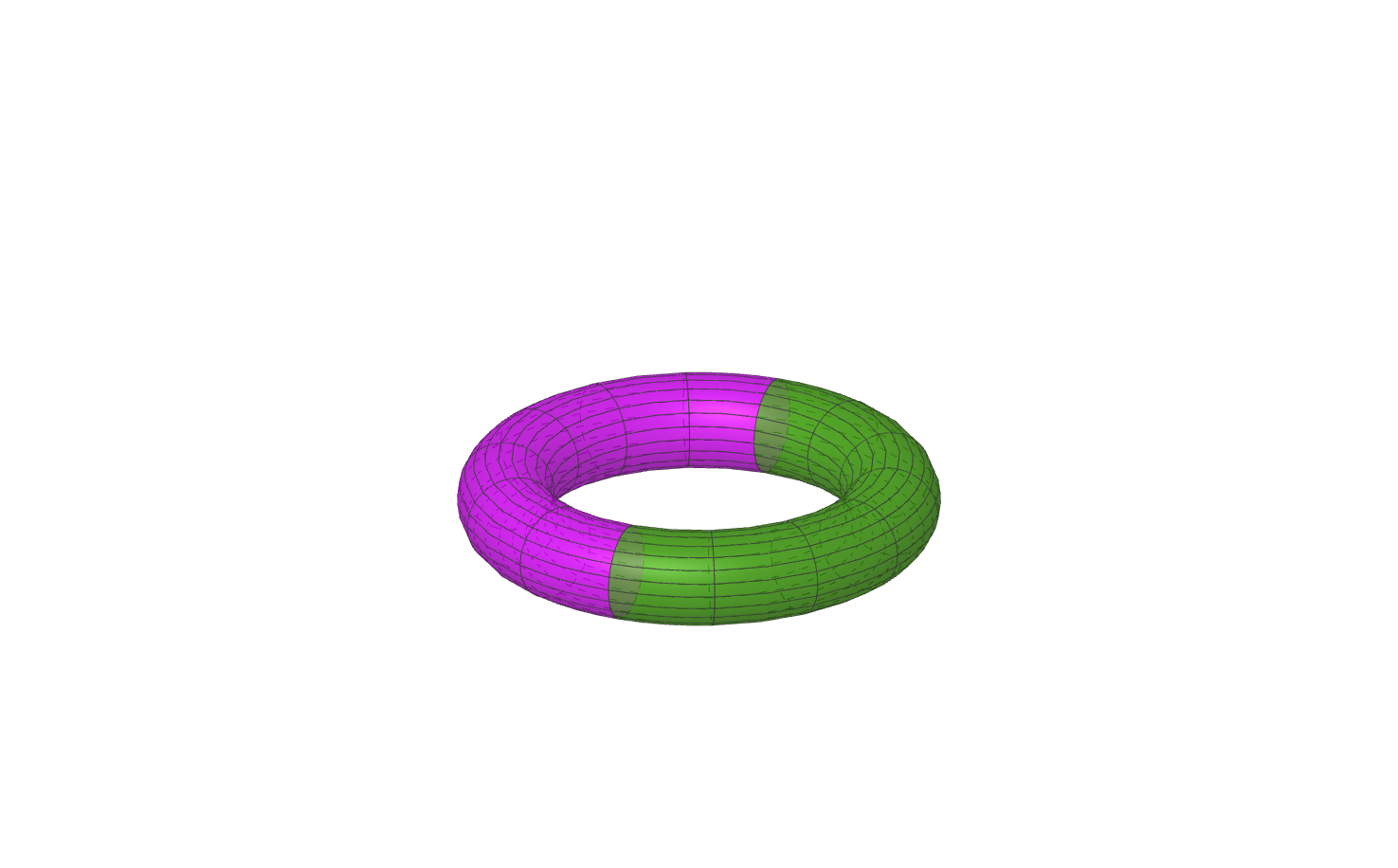}\label{fig:Torus}}
\caption{Examples of Cheeger sets that are (a) a strict mass constrained perimeter minimizer and (b) an integral mass constrained perimeter minimizer.}
\end{figure}


With these definitions in hand, we now state the isoperimetric stability result that we will utilize in this paper.

\begin{proposition}
\label{prop:Engelstein}
\cite[Lemma 3.4]{chodosh2019riemannian}
Suppose that $E^*\subset \M$ is a global, mass-constrained perimeter minimizer with mass $\vv$ and is both i) smooth (meaning it has empty singular set) and ii) either a strict perimeter minimizer or integrable.
Then, there exists constants $c,\delta >0$ such that for any measurable $E \subseteq \M$ with $\vol_\M(E) = \vv$ and $\vol_\M(E \Delta E^*) < \delta$ one has 
\begin{equation}
\label{eqn:isoper-stability}
\Per(E) - \Per(E^*) \geq c \alpha(E)^2.
\end{equation}
\end{proposition}

One heuristic way of describing this result, in the spirit of the viewpoint given in~\cite{white1994strong}, is that the mass-constrained perimeter functional is strongly convex with respect to the $\Lp{1}$ norm.
We will utilize this result to show that any set which has small ``perimeter deficit'', i.e. which has perimeter close to a mass-constrained minimizer, must be close to a mass-constrained perimeter minimizer in the sense of Fraenkel distance, which is an $\Lp{1}$ type distance, with a clearly quantified estimate.
\add{ It is worth noting that the result in \cite{chodosh2019riemannian} is significantly more general than the one given here. Indeed, in their work they also prove that for analytic manifolds there always exists some exponent $\gamma_\M$, so that the previous proposition holds without assuming strict or integrable minimizers, under the modification that we replace $2$ with $\gamma_\M$. The analysis in our paper would continue to work in this more general setting, at the cost of now having an exponent that depends upon the manifold. For concreteness of presentation, we have opted to state the result in terms of a single exponent (i.e. $2$) which does not depend upon the manifold $\M$, but at the cost of making additional assumptions on the minimizing Cheeger set.}

\begin{example}
\label{ex:1D}
\add{To better understand how one derives the estimates, it is useful to consider a toy example. Consider a domain $S = (-1,1) \times (-M/2,M/2)$, with $M$ large, and let us consider the isoperimetric problem with mass constrained to be equal to $M$. In that case one perimeter minimizer is given by the set $E^* = \{(x,y) \in S \: : \:y<0, \quad x \in(-1,1)\}$. Now, consider a set $E = \{ (x,y) \in S \: : \: y \leq g(x), \quad x\in(-1,1)\}$ for some smooth $g$, which has mean zero and is compactly supported on $(-1,1)$.
We can then write
\[ \Per(E;S) - \Per(E^*;S) = \int_{-1}^1 \sqrt{1 + (g')^2} -1 \,\dd x. \]
The goal is to show that this quantity is greater than $c(\int |g|\,\dd x)^2$.}

Restricting to the situation where $g'$ is relatively small (which is the critical case for proving the inequality), and using a Taylor expansion, we find the bound:
\[ \Per(E) -\Per(E^*) \geq c \int_{-1}^1 (g')^2 \,\dd x. \]
In turn, using Poincare's and H\"older's inequality
\[ \Per(E) - \Per(E^*) \geq c\|g\|_{\Lp{2}}^2 \geq c\|g\|_{\Lp{1}}^2 = c\vol_\M(E \Delta E^*)^2. \]
\end{example}

The extension of the previous example to graphs of smooth functions, centered around smooth surfaces is entirely analogous.
The biggest technical challenge in establishing results such as Proposition~\ref{prop:Engelstein} is that one a-priori cannot reduce to the case where $E$ is locally expressed as the graph of a smooth function, see further discussion at the end of the section.
\vspace{\baselineskip}

In this work, we will be concerned with global minimizers of the continuum Cheeger energy, and hence with global mass-constrained perimeter minimizers.
Thus we will rely on the following, slightly modified version of the previous proposition

\begin{proposition}
\label{prop:global-perimeter-minimizers}
Fix a mass $0 < \vv < 1$.
Suppose that any global, mass-constrained perimeter minimizer with mass $\vv$ is both i) smooth (meaning it has empty singular set) and ii) either a strict perimeter minimizer or integrable.
Then, there exists a $c>0$ so that for any $\vol_\M(E) = \vv$ we have
\[ \Per(E) - \mathbb{I}(\vv) \geq c\alpha(E)^2. \]
\end{proposition}

\begin{proof}
Suppose, for the sake of contradiction, that no such $c$ exists.
Then there exists a sequence of sets $E_k$ satisfying $\vol_\M(E_k)=\vv$ so that
\[ \frac{\Per(E_k) - \mathbb{I}(\vv)}{\alpha(E_k)^2} \to 0. \]
By $\Lp{1}$ compactness of sets of finite perimeter, after taking a subsequence (not relabeled) we have that there exists a set $\tilde E$ satisfying $\vol_\M(E_k \Delta \tilde E) \to 0$ and $\mathbb{I}(\vv) = \liminf_k \Per(E_k) \geq \Per(\tilde E) \geq \mathbb{I}(\vv)$.
This implies that $\tilde E$ is a mass-constrained perimeter minimizer, and hence satisfies the assumptions of Proposition~\ref{prop:Engelstein}.
But this then implies that, for $k$ large enough, $\frac{\Per(E_k) - \mathbb{I}(\vv)}{\alpha(E_k)^2} > c>0$, where $c$ is chosen as in that proposition.
This is a contradiction, and concludes the proof.
\end{proof}

\subsubsection{Convergence Rates for Cheeger Cuts: Results}
\label{subsubsec:problemsetup:CheegerCutsRes}

With tools from the previous section in hand, we now state a final main result of our work.
Before we do so, we state in detail the technical assumptions that we make on minimizers of the continuum Cheeger problem (discussion of these assumptions is given after the statement of the theorem).

\begin{assumption}
\label{assumption:Cheeger-sets}
(Assumptions on Cheeger sets of $\M$)
\begin{enumerate}
\item (Local strong convexity of isoperimetric \rem{function} \add{profile}) First, we define $\capV = \argmin \frac{\mathbb{I}(\vv)}{\min(\vv,1-\vv)}$.
We assume that for every $\vv \in \capV$ there exists a $C>0$ and an $\eta$ so that for any $\vv' \in (\vv-\eta,\vv+\eta)$ we have that the function $g(\vv) = \frac{\mathbb{I}(\vv)}{\min(\vv,1-\vv)}$ satisfies $g(\vv') \geq g(\vv) + C(\vv-\vv')^2$.
\item Second, we assume that all of the minimizers of the mass constrained isoperimetric problem with masses in $\capV$ satisfy all of the assumptions in Proposition~\ref{prop:global-perimeter-minimizers}, meaning that they are at least $\Ck{3}$ and are either strict or integrable.
\end{enumerate}
\end{assumption}

\add{As in the case of the assumptions in Proposition \ref{prop:Engelstein}, it is likely possible to weaken Assumption \ref{assumption:Cheeger-sets} i) as long as one is willing to permit the exponent in the results to depend upon the manifold. In particular, if the function $\frac{\mathbb{I}(\vv)}{\min(\vv,1-\vv)}$ can be bounded from below near its minimizers using a different exponent, then all of the analysis that we give here would continue to apply using that exponent. We conjecture that for any analytic manifold Assumption \ref{assumption:Cheeger-sets} i) holds as long as one is willing to replace $2$ with a manifold dependent exponent. Again, as in the case of Proposition \ref{prop:Engelstein}, we have opted to work in terms of a concrete (but natural), manifold independent exponent, at the cost of making additional assumptions on the manifold.}

Using these assumptions we will then establish our main result regarding the convergence of Cheeger sets.

\begin{theorem}
\label{thm:rates-Cheeger-Cuts}
(Asymptotic consistency of Cheeger cuts, with rates)
\rem{Suppose that Assumptions \ref{assumptions} and \ref{assumption:Cheeger-sets} hold, and suppose that $\delta,\theta,\zeta>0$ are small enough quantities and in particular $\delta \leq \veps/4$.
Then, there are constants $C_1, C_2, C,c>0$ (that depend on $\M$) such that with probability at least $1- n \exp(-cn\theta^2\delta^m)-C \exp\left(-cn\zeta \e^{\frac{m+1}{2}}\right)$ for any minimizer $E_n^*$ of the discrete Cheeger energy there exists a minimizer $E^*$ of the continuum Cheeger problem and a map $T_n$ from $\M$ to $\M_n$ which satisfy }
\add{Let $\M$ and $\veps$ satisfy Assumptions~\ref{assumptions} and $\M$ satisfy Assumptions~\ref{assumption:Cheeger-sets}.
Then, there exists constants (that may depend on $\M$) $\theta_0,\zeta_0,C_1,C_2,C,c,c'>0$ such that for any $\delta,\theta,\zeta>0$, with $n\zeta\eps^{\frac{m+1}{2}}\geq c$, $\delta\leq \frac{\veps}{4}$ , $ c' \log(n)/n \leq \theta^2 \delta^m $, $\theta\leq \theta_0$ and $\zeta\leq \zeta_0$, we have} \addii{that}\add{, with probability at least $1- n \exp(-cn\theta^2\delta^m)-C \exp\left(-cn\zeta \e^{\frac{m+1}{2}}\right)$, for any minimizer $E_n^*$ of the discrete Cheeger energy there exists a minimizer $E^*$ of the continuum Cheeger problem and a map $T_n$ from $\M$ to $\M_n$ which satisfy}
\begin{equation}
\label{eqn:Cheeger-estimate}
\lVert \one_{E_n^*} \circ T_n  - \mathds{1}_{E^*} \rVert_{\Lp{1}(\M)} \leq C\kappa^{\frac{m-1}{4m}}, \qquad \|T_n - \Id\|_{\Lp{\infty}(\M)} \leq \delta,
\end{equation}
where
\[ \kappa := \sqrt[6]{\eps} + \frac{\delta}{\e} + \theta + \zeta. \]
\end{theorem}

\begin{remark}
Again, here we are permitted to make choices of the various parameters in order to determine convergence rates as a function of $n$.
In particular, if we pick our parameters as in Remark~\ref{rem:Cheeger-rates} we get that, with high probability and after neglecting logarithms, $\kappa = n^{-\frac{1}{4(1+2m)}}$.
This in turn gives an $\Lp{1}$ convergence rate of the Cheeger cuts of order $n^{-\frac{(m-1)}{16m(1+2m)}}$.
\end{remark}

\begin{remark}
As for our rates established for the Cheeger constant, we believe that the rate of convergence for Cheeger cuts we achieve are not optimal.
If, as in Remark~\ref{rem:CheegerConsRateOpt}, we compare with the eigenvalue problem~\eqref{eq:CFWEigenPair} then~\cite{garciatrillos19} establishes a convergence rate of $n^{-\frac{1}{4m}}$ (ignoring logarithms) for eigenvectors of the graph Laplacian, to be compared with an approximate rate of $n^{-\frac{1}{32m}}$ we establish in Theorem~\ref{thm:rates-Cheeger-Cuts}.
In particular, we have likely given up something in terms of rates in several ways when applying the mass-constrained isoperimetric stability results (see Remark \ref{rem:not-sharp}). 
We do note that the stability estimates for isoperimetric inequalities are, in terms of their exponent, known to be sharp.
In any case, demonstrating sharpness of convergence rates, given the highly non-linear nature of the optimization problem that we analyze, would require completely different tools.
\end{remark}

\begin{remark}
\label{rem:OtherConv1}
A different type of convergence for minimizers can be deduced from Theorem \ref{thm:rates-Cheeger-Cuts}, namely, for $E^*$ in \eqref{eqn:Cheeger-estimate} with smooth boundary we can prove that (with probability at least $1- n \exp(-cn\theta^2\delta^m)-C \exp\left(-cn\zeta \e^{\frac{m+1}{2}}\right)$):
\[  \nu_n( E^* \Delta E_n^* ) \leq   C \kappa^{\frac{m-1}{4m}} + C \delta. \]
We show in Remark \ref{rem:OtherConv2} how this inequality can be obtained from the analysis we develop for our proof of Theorem \ref{thm:rates-Cheeger-Cuts}.	
\end{remark}

We make a few remarks about the statement of Theorem \ref{thm:rates-Cheeger-Cuts}. 
The first condition in Assumption~\ref{assumption:Cheeger-sets} is to guarantee that minimizers of the Cheeger problem are in some sense strict minimizers, at least with respect to variations in volume.
We remark that this condition, along with well-known asymptotics~\cite{morgan2000some} for the isoperimetric \rem{function} \add{profile} near zero and compactness of the perimeter functional, imply that $\capV$ contains at most finitely many elements. The assumption that minimizers have smooth boundary is natural and there are reasons to believe that this holds for all minimizers in certain classes of manifolds (see the discussion below). 

We notice that Theorem \ref{thm:rates-Cheeger-Cuts} only proves closeness to some Cheeger set $E^*$: the case of the torus in Example~\ref{ex:torus} makes it readily clear why this is necessary.
We also remark that although the theorem was stated in terms of the global minimizer of the discrete energy, we could prove analogous bounds for approximate minimizers of the discrete Cheeger energy.
More concretely, if $J_n(E_n):= \frac{\GTV_{n\eps}(\one_{E_n})}{\min\{\nu_n(E_n),1-\nu_n(E_n)\}}$ and we knew that $J_n(E_n) - \inf J_n < \gamma_n \to 0$, then we could provide an analogous estimate with a right hand side that depends upon $\gamma_n$.

In this paper we have focused exclusively on proofs for the Cheeger problem.
However, the techniques are extendable, \textit{mutatis mutandis}, to other regularized cut problems such as the ratio cut in~\eqref{eq:RatioCut} and graph modularity clustering in \eqref{def:ModularityClust}.
In particular, condition i) in Assumptions \ref{assumption:Cheeger-sets} would need to be appropriately adjusted to each cut problem.

\subsection{Discussion of Literature for Isoperimetric Stability Problems}
\label{subsec:problemsetup:IsoStabRev}

There is a long history of the study of the stability of isoperimetric problems.
Early works established these types of inequalities in the case of smooth sets in $\R^2$~\cite{bonnesen1924isoperimetrische}, or for smooth sets in $\R^d$~\cite{hall1992quantitative}.
Another important early work was~\cite{white1994strong}, which proposed the principle that any set which is stable in terms of second variation of the (mass-constrained) perimeter should be a local minimizer of the mass-constrained isoperimetric problem in a local, $\Lp{1}$ sense.
However, their argument did not provide quantified estimates of this minimality.
In the last ten years, many works have sought to quantify this relationship.
The first of these works,~\cite{fusco2008sharp} established the quantified relationship in the case of the ball in $\R^d$.
A crucial element of their proof was the observation that by symmetrizing the set of interest (in the sense of a Steiner symmetrization) one can reduce the perimeter of the set without increasing its Fraenkel asymmetry by too much.
After symmetrizing several times, one can reduce to the type of one-dimensional estimate that we presented in Example~\ref{ex:1D}.
A later, elegant alternative proof for the ball was obtained by using techniques from optimal transportation~\cite{figalli2010mass}.

Isoperimetric problems in $\R^d$ turn out to be central to a number of problems in ``hard analysis''.
In particular, these estimates have been used to provide quantitative estimates of the stability of Sobolev~\cite{cianchi2009sharp}, Polya-Szeg\H{o}~\cite{cianchi2008quantitative} and other functional inequalities~\cite{fusco2009stability}.

The extension of the stability results for isoperimetric problems on $\R^d$ to more general settings is the focus of current research.
These problems are often quite challenging, because the loss of explicit minimizers and symmetries renders many of the technical tools unusable.
However, an alternative approach to stability for the classical isoperimetric problem is given \rm{in~\cite{cicalese12}.} 
There they study the quotient $\frac{\Per(E)-\Per(B)}{\alpha(E)^2}$, and demonstrate that for a fixed value of the denominator, the minimizer of the numerator is very well behaved (i.e. regular, and locally expressible as the graph of a function).
The procedure of selecting minimizers of this quotient, and then studying their properties is known as a \emph{selection principle}, and allows one to reduce the study of isoperimetric stability to the study of relatively simple sets.
Technically, this reduction is accomplished by studying penalized variational problems, and using the regularity theory of minimizers of penalized isoperimetric problems~\cite{tamaninni1984regularity}.
Using this reduction, one can then establish the quantitative isoperimetric inequality using standard techniques~\cite{fuglede1989stability}, which are very similar to those we described in Example~\ref{ex:1D}.
The recent work~\cite{chodosh2019riemannian}, which we rely on in this paper, extends the techniques \rm{in~\cite{cicalese12}} 
and~\cite{fuglede1989stability} to the setting of Riemannian manifolds.

\add{Very recently \cite{cicalese2020maximal} utilizes similar ideas as the second part of this paper within the context of discrete to continuum bounds for crystallization energies. In particular, they study an (anisotropic) analog of the perimeter on certain classes of periodic lattices. In their work they use stability of the limiting anisotropic isoperimetric problem, proven in \cite{figalli2010mass} and analogous to \cite{chodosh2019riemannian} in our setting, in order to bound a certain distance between discrete configurations and the continuum \addii{wulff} set, in terms of sort of energy difference. This is directly analogous to the work that we do in Section \ref{sec:CheegerCutsProof}, but in the context of a different energy} \add{as well as different motivation}.

Central to many of these works is the question of regularity of minimizers for isoperimetric problems.
The study of this problem is very classical.
Necessary conditions for the mass-constrained isoperimetric problem in $\R^d$ require that, if solutions are smooth, they must be surfaces of constant mean curvature which are $\Ck{\infty}$ up to a set of measure $d-8$~\cite{GMT83,Gruter}.
There are examples of surfaces which possess constant mean curvature everywhere except on a set of dimension $d-8$~\cite{simons1968minimal}. \add{In the context of isoperimetric problems on convex domains there are recent conjectures, which posit that such surfaces cannot be minimizers of the isoperimetric problem~\cite{jerison2019two}. This conjecture is informed by earlier work in the context of convex domains~\cite{SternbergZumbrunIsoPer}, which strongly suggests that there are topological obstructions which prevent the example in~\cite{simons1968minimal} from being a global perimeter minimizer. Of course convex domains are in a sense much simpler than the manifolds we consider in this paper, for which, to our knowledge, very little is known about whether global minimizers to isoperimetric problems can be guaranteed to be regular. This then means that we can only be certain that} \add{the assumptions in Theorem \ref{thm:rates-Cheeger-Cuts} are met, in general, in the case where $\M$ is of small dimension.}

In summary, there is a rich mathematical theory studying the minimizers of isoperimetric problems, and the stability of such minimizers.
Our paper draws on this theory to provide new quantitative tools for the asymptotic consistency of statistical problems based upon geometric optimization.


  

\section{Quantitative Estimates for Non-Local Operators}
\label{sec:nonlocal}

In this section we establish some results concerning the non-local TV seminorm~\eqref{eqn:NonLocalTvseminorm}.
These results are extensions to the manifold case of somewhat well known results in the flat Euclidean case.
The non-local TV seminorm is analogous to the non-local Dirichlet energy studied in~\cite{BuragoIvanovKurylev}.
In form, these two functionals differ only in the powers used for the integrand $|f(x) - f(y)|$ and in their rescaling factors, but their properties are markedly different.
The study of non-local Dirichlet energies was a fundamental piece in~\cite{BuragoIvanovKurylev,garciatrillos19} to obtain quantitative rates for the spectral convergence of graph Laplacians.

To fix some ideas, we first introduce a collection of definitions from differential geometry that will provide us with the right language to formalize our arguments.
For the most part the notions introduced below will suffice for proving most of our estimates: the only exception is Proposition~\ref{prop:ConvolutionMonotone} where more tools are needed.
These will be presented in the Appendix. 

For a given $x \in \M$ we use $\exp_x$ to denote the \emph{exponential map} at $x$.
This is a diffeomorphism between the set $B(0,h) \subseteq \T_x\M$ and $B_\M(x,h)$ as long as $h \leq i_\M$ (where we recall $i_\M$ is $\M$'s injectivity radius), and is characterized as follows: for $v\in B(0,h) \subseteq \T_x\M$, the curve $t \in [0,1] \mapsto \exp_x(tv)$ describes the constant speed geodesic emanating from $x$ with initial velocity $v$.
We use $\dd (\exp_x)_v$ to denote the differential of $\exp_x$ at $v \in B(0, h) \subseteq \T_x\M$, which maps tangent vectors to $\T_x\M$ at $v$ to tangent vectors to $\M$ at $\exp_x(v)$.
Implicit in this definition of exponential map is the choice of \textit{connection} or \textit{covariant derivative} which here we take it to be the Levi-Civita connection; recall that the metric on $\M$ is the one inherited from the ambient space $\R^d$.
We use $J_x(v)$ to denote the \textit{Jacobian} of the exponential map at the point $v$.
The Jacobian describes the \add{Riemannian} volume form of the manifold $\M$ (which we will denote by $\vol_\M$) in the local coordinates of the exponential map (a.k.a. normal coordinates).
That is, we can write integrals of the form
\[ \int_{B_\M(x, h)} \zeta(y) \, \dd\vol_\M(y) \]
(for small enough $h$) as
\[ \int_{B(0, h) \subseteq \T_x\M} \zeta(\exp_x(v))J_x(v) \, \dd v. \]
According to~\cite{BuragoBook,BuragoIvanovKurylev}, the Jacobian $J_x$ satisfies the following bounds 
\begin{equation}
\label{eqn:Jacobian}
1- C |v|^2 \leq J_x(v) \leq 1+ C |v|^2,  
\end{equation}
for all $v\in B(0, h) \subseteq \T_x\M$, where $h$ is sufficiently small, and in particular smaller than $i_\M$.
The constant $C$ can be written as $C = cm K$, where $c$ is a universal constant, $m$ is the dimension of $\M$, and $K$ is an upper bound on the absolute value of all sectional curvatures of $\M$ at all points within the ball $B_\M(x,h)$.
Since we are assuming $\M$ to be compact, this constant can be picked to uniformly bound the discrepancy between the volume form (in normal coordinates) and the uniform measure.
In the remainder we will write $h \leq h_\M$ to indicate that $h$ is small enough (smaller than a fixed quantity that only depends on $\M$).
We define the tangent bundle by
\[ \T\M := \lb (x,v) \: : \: x \in \M , \quad v \in \T_x\M \rb. \]
Recall that we wrote the non-local $\TV$ seminorm $\TV_h$ as a double integration over $x\in\M$ and $y\in\M$, where $x$ is close to $y$, see~\eqref{eqn:NonLocalTvseminorm}.
In Euclidean settings we can use a change of variables and integrate over $x$ and $v=\frac{x-y}{h}$.
With this change of coordinates the non-local functional converges easily to a local functional.
On a manifold we can't immediately make the same transformation as we need to be careful where $v$ lives.
However, this is really just a technical detail and (formally) by viewing $v$ as a tangent vector in $\T_x\M$ we can again view the double integration as a integral over $x\in\M$ and $v\in\T_x\M$.
This leads us to requiring the volume form on $\T\M$, i.e. $\dd \vol_{\T\M}(x,v)$.
We give some background on how one rigorously defines $\dd \vol_{\T\M}$ in the appendix where it is needed for a more technical computation, but within the main body it is enough to understand that $\int_{\T\M}\zeta(\xi)\,\dd\vol_{\T\M}(\xi) $ can be written as
\[ \int_{\T\M}\zeta(x,v)\,\dd\vol_{\T\M}(x,v) = \int_\M \int_{B(0, h) \subseteq \T_x\M} \zeta(x,v)\, \dd v \dd\vol_\M(x) \]
for any $\zeta:\T\M\to \R$ with the property that $\zeta(x,\cdot)$ has compact support in $B(0,h_\M)$ for every $x\in\M$.


We will use the \emph{geodesic flow} $\Phi= \{ \Phi_s \}_{s \in [0,1]}$ which takes an arbitrary point 
\[(x,v) \in \mathcal{B} := \{ (x,v)\: : \: x \in \M, \quad v\in \T_x \M , \quad |v|_x < i_\M \}, \]
into the point
\[ \Phi_s(x,v):= (\Phi^1_s(x,v), \Phi_s^2(x,v))\]
where
\[ \Phi_s^1(x,v):= \exp_x(sv), \quad \Phi_s^2(x,v):= \dd(\exp_x)_{sv}(v). \]
It can be checked that for every $s\in [0,1]$, $\Phi_s$ is a diffeomorphism of $\mathcal{B}$ into itself. Moreover, the geodesic flow leaves $\vol_{\T \M}$ invariant: this is the content of Liouville's theorem (see Chapter 3 in \cite{docarmo1992riemannian}). 

In the sequel we may use $\xi=(x,v)$ to represent a generic point in the tangent bundle $\T\M$ and abuse notation slightly to write things like $ g(\xi)=g(x)$ whenever $g$ is a real valued function on $\M$.
We will also use $\dd g$ (i.e. the differential of $g$) which is the 1-form that when acting on a tangent vector $v$ returns the directional derivative of $g$ in the direction $v$, and will write things like $\dd g(\xi)$ to denote the directional derivative of $g$ at the point $x$ in the direction $v$. If a function $g\in\Ck{1}(\M)$ then we understand it to be differentiable in the sense of Fr\'echet, in particular, at every $x\in \M$ there exists $\nabla g(x)\in \T_x\M$ such that $\dd g(v)(x) = \langle \nabla g(x), v\rangle_x$ for all $v\in \T_x\M$.

With all the above definitions in hand, we are ready to state and prove our first auxiliary results.

\begin{proposition}
\label{prop:biasineq}
There is a constant $C$ such that for all $f \in \BV$\addii{$(\M)$} \add{(i.e. $f \in \Lp{1}(\M)$ and $\TV(f) < \infty$)} and all $0<h\leq h_\M$  we have 
\[ \TV_{h}(f) \leq (1+ Ch^2)\sigma_\eta \TV(f), \]
where $\sigma_\eta$ is the surface tension defined in~\eqref{def:sigmaeta}.
\end{proposition}

\begin{proof}
Using the density of smooth functions in $\Lp{1}(\M)$ and an approximating result for the $\TV$ seminorm like that in Theorem 13.9 in \cite{Leoni} (see also Theorem 2.4 in \cite{ambrosio2015bv}), it is enough to show the result for $f \in \Ck{\infty}(\M)$.
Let $x\in \M$ and $y \in B_\M(x,h)$.
Then, there is a unique $v \in B(0,h) \subseteq \T_x\M$ such that $\exp_x(v)=y$.
By the fundamental theorem of calculus we can write
\[ f(y) - f(x) = f(\exp_x(v))- f(x)= \int_{0}^1 \frac{\dd}{\dd t}f(\exp_x(tv))\,\dd t \]
and so 
\begin{align*}
\begin{split}
\int_{B_\M(x, h )} |f(y)- f(x)| \, \dd \vol_\M(y) & = \int_{B(0,h) \subseteq \T_x\M} |f(\exp_x(v))- f(x)| J_x(v) \, \dd v \\
& \leq (1+ C h^2) \int_{B(0,h) \subseteq T_x \M} \int_{0}^1 |\dd f(\Phi_t(x,v))| \, \dd t \, \dd v,
\end{split}  
\end{align*}
where we have used~\eqref{eqn:Jacobian} to bound the Jacobian.
It follows that
\begin{align*}
\begin{split}
\int_\M \int_{B_\M(x, h)} |f(y)- f(x)| \, \dd \vol_\M(y) \, \dd \vol_\M(x) \leq (1+ Ch^2) \int_{0}^1 \int_{\mathcal{B}_h } |\dd f(\Phi_t(\xi))| \, \dd\vol_{T\M}(\xi) \, \dd t
\end{split}
\end{align*}
where 
\[ \mathcal{B}_h := \{ (x,v) \in \T\M \: : \: x \in \M , \quad v\in B(0,h) \subseteq \T_x\M \}. \]
From the fact that the geodesic flow leaves $\vol_{\T\M}$ invariant, and the fact that $\Phi_t(\mathcal{B}_h) = \mathcal{B}_h$, it follows after a change of variables that for all $t \in (0,1)$ 
\begin{align*}
\int_{\mathcal{B}_h } |\dd f(\Phi_t(\xi))| \, \dd\vol_{\T\M}(\xi) & = \int_{\mathcal{B}_h} |\dd f(\xi)| \, \dd\vol_{\T\M}(\xi) \\
& = \int_\M\int_{B(0,h) \subseteq \T_x\M} |\langle \nabla f(x),v\rangle_x| \, \dd v \, \dd \vol_\M(x) \\ 
& = \int_{\M} |\nabla f(x)|_x \int_{B(0,h) \subseteq \T_x\M} \left|\left\langle \frac{\nabla f(x)}{|\nabla f(x)|_x},v \right\rangle_x\right| \, \dd v \, \dd \vol_\M(x) \\
& = \int_\M \sigma_\eta |\nabla f(x)|_x \, \dd \vol_\M(x),
\end{align*}
where in the last line we have used the radial symmetry of the integrand and the definition of $\sigma_\eta$ in~\eqref{def:sigmaeta}. From the above it follows that
\begin{align*}
\int_\M \int_{B_\M(x,h)} |f(y)- f(x)| \, \dd\vol_\M(y) \, \dd \vol_\M(x) \leq (1+Ch^2)\sigma_\eta \int_{\M}|\nabla f(x)|_x \, \dd\vol_\M(x). 
\end{align*}
Using \eqref{eqn:smoothf} (and the fact that $f$ is smooth) we deduce the desired inequality.
\end{proof}

The next result is a somewhat converse to the previous one, but is restricted to smooth enough functions.

\begin{proposition}
\label{prop:NonLocalLocalRegu}
(Non-local vs local for smooth functions)
Let $0 < h\leq h_\M$ and let $f$ be a \add{$\Ck{1,1}(\M)$} function. 
Then, 
\[ \sigma_\eta \TV(f) \leq (1+Ch^2)\TV_h(f) + C\lVert f \rVert_{\Ck{1,1}(\M)} h, \]
where $C$ is independent of $f$ or $h$.
\end{proposition}

\begin{proof}
Let $f\in\Ck{1,1}(\M)$.
For a fixed $x\in \M$ we can Taylor expand $f$ around $x$ and write 
\[ f(y) = f(x) + \langle \nabla f(x),\exp_x^{-1}(y) \rangle_x + R_x(y), \quad y\in B_\M(x,h) \] 
where the remainder $R_x(y)$ satisfies
\[ \sup_{y \in B_\M(x,h)} |R_{x}(y)| \leq C\lVert f \rVert_{\Ck{1,1}(\M) }h^2 \]
for a constant $C$ that only depends on $\M$ (and in particular does not depend on $f$).
It follows that
\begin{align*}
& \frac{1}{h^{m+1}} \int_\M |f(x) - f(y)| \eta\left(\frac{d_\M(x,y)}{h}\right) \, \dd\vol_\M(y) + C h \lVert f \rVert_{\Ck{1,1}(\M)} \\
& \hspace{1cm} \geq \frac{1}{h^{m+1}} \int_\M \la \left\langle \nabla f(x),\exp_x^{-1}(y) \right \rangle_x \ra \eta\left(\frac{d_\M(x,y)}{h} \right) \, \dd\vol_\M(y).
\end{align*}
The term on the right hand side of the above expression can be written as 
\begin{align*}
& \frac{|\nabla f(x)|_x}{h^{m+1}}\int_{B(0,h) \subseteq \T_x\M} \left| \left\langle \frac{\nabla f(x)}{|\nabla f(x)|_x},v\right\rangle_x \right| \eta\left(\frac{|v|}{h}\right) J_x(v) \, \dd v \\
& \hspace{1cm} \geq (1-Ch^2)\frac{|\nabla f(x)|_x}{h^{m+1}}\int_{B(0,h) \subseteq \T_x\M} \left| \left\langle \frac{\nabla f(x)}{|\nabla f(x)|_x},v\right\rangle_x \right| \eta \left(\frac{|v|}{h}\right) \, \dd v \\
& \hspace{1cm} = \sigma_\eta (1- Ch^2) \lvert \nabla f(x) \rvert_x,
\end{align*}
where in order to go from the first to the second line we have used~\eqref{eqn:Jacobian}, and where in the last line we have used the radial symmetry of the integrand and the definition of $\sigma_\eta$ in~\eqref{def:sigmaeta}.
Integration over $x$ gives us the desired inequality.
\end{proof}

After considering the relationship between local and non-local energies we now present some results that relate non-local energies at different length-scales.
First we prove a subadditivity property.

\begin{lemma}
\label{lemm:subadditivity}
(Subadditivity)
Let $A$ be a Borel subset of $\M$ and let $g_A: (0,h_\M) \rightarrow \R$ be the function given by
\[ g_A(h):= \int_{\T\M} \eta(|v|_x)\one_A(x) \one_{A^c}(\Phi_h^1(x,v)) \, \dd \vol_{\T\M}(x,v). \]
Then, for every $h,h'$ with $h+h' < h_\M$ we have
\[ g_A(h+h') \leq g_A(h) + g_A(h'). \]
In particular, if $N \in \N$ is such that $Nh$ is smaller than $h_\M$ we can iterate the previous identity to obtain
\[ g_A(Nh) \leq Ng_A(h). \]
\end{lemma}

\begin{proof}
First of all notice that for every $(x,v) \in \T\M$ we have
\[ \one_{A}(x)\one_{A^c}(\Phi^1_{h+h'}(x, v))\leq  \one_{A}(\Phi^1_{h}(x,v))\one_{A^c}(\Phi_{h+h'}^1(x,v))+ \one_{A}(x)\one_{A^c}(\Phi^1_{h}(x,v)). \]
Indeed, if the left hand side is equal to one, necessarily one of the two terms on the right hand side must be equal to one.
From this we conclude that \add{
\begin{align*}
\begin{split}
g_A(h+h') & \leq  \int_{\T \M} \eta(|v|_x) \one_{A}( \Phi^1_h(x,v)) \one_{A^c}( \Phi^1_{h+h'}(x,v)) \, \dd \vol_{\T\M}(x,v) \\
&  \hspace{0.3cm} + \int_{\T\M} \eta(|v|_x) \one_A(x) \one_{A^c}( \Phi^1_h(x,v)) \, \dd \vol_{\T\M}(x,v)  \\
& = g_A(h)
\\ & \hspace{0.3cm}  +\int_{\T\M} \eta(|\Phi^2_h(x,v)|_{\Phi^1_h(x,v)}) \one_{A}( \Phi^1_{h}(x,v)) \one_{A^c}( \Phi^1_{h'}\left(\Phi^1_{h}(x,v), \Phi^2_h(x,v) \right) ) \, \dd \vol_{\T\M}(x,v),
\end{split}
\end{align*}
}
where the last equality follows from the fact that $|v|_x =|\Phi^2_h(x,v)|_{\Phi^1_h(x,v)}$ (because lengths are preserved along geodesics) and the fact that $\Phi^1_{h+h'}(x,v)=\Phi^1_{h'}\left(\Phi^1_{h}(x,v), \Phi^2_h(x,v) \right)$ (which essentially says that moving along the geodesic starting at $x$ in the direction $v$ for $h+h'$ units of time is the same as moving only for $h$ unites of time and then continue moving for an extra $h'$ units of time).
To conclude, we consider the change of variables 
\[ (x,v) \mapsto (x',v') =\Phi_{h}(x,v) \]
and rewrite the last integral as
\[ \int_{\T\M} \eta(|v'|_{x'}) \one_{A}(x')\one_{A^c}( \Phi_{h'}(x',v')) \, \dd \vol_{\T\M}(x',v')= g_A(h'), \]
using the fact that the geodesic flow leaves the volume form $\vol_{\T\M}$ invariant according to Liouville's theorem.
This proves the result.
\end{proof}

We can now use the previous lemma to prove that the non-local TV seminorm with a small length-scale dominates the ones with larger length-scale. 

\begin{proposition}
\label{prop:monotonTV}
There is a constant $C>0$ such that for every $f \in L^1(\M)$ and every $ h, a$ with $0< h \leq a < \frac{h_{\M}}{2}$ we have
\[ \TV_{a}(f) \leq C \TV_{h}(f). \]
\end{proposition}

\begin{proof}
We first restrict our attention to the case $f= \one_A$ for some Borel subset $A$ of $\M$.
Notice that for arbitrary $h \leq h_\M$ we can rewrite $\TV_{h}(\one_A)$ as:
\begin{align*}
\begin{split}
\TV_h(\one_A) := & \frac{1}{h^{m+1}}\int_{\M}\int_\M \eta \left( \frac{d_\M(x,y)}{h} \right) |\one_A(x) - \one_{A}(y)| \, \dd \vol_\M(y) \, \dd \vol_\M(x) \\
& = \frac{2}{h^{m+1}}\int_{\M}\int_\M \eta \left( \frac{d_\M(x,y)}{h} \right) \one_A(x) \one_{A^c}(y) \, \dd \vol_\M(y) \, \dd \vol_\M(x) \\
& = \frac{2}{h^{m+1}} \int_{\M}\int_{B(0,h)\subseteq\T_x\M} \eta \left(\frac{|v|_x}{h} \right)\one_{A}(x)\one_{A^c}(\exp_x(v)) J_x(v) \, \dd v \, \dd \vol_\M(x) \\
& = \frac{2}{h} \int_{\M}\int_{B(0,1)\subseteq\T_x\M}\eta \left(|v|_x \right)\one_{A}(x)\one_{A^c}(\exp_x(h v)) J_x( h v) \, \dd v \, \dd \vol_\M(x).
\end{split}
\end{align*}
Now, from~\eqref{eqn:Jacobian} we know that for all $x \in \M$ and $v \in \T_x\M$ with $|v| \leq 1$ we have
\[ 1-Ch^2 \leq J_{x}(h v) \leq 1+ C h^2 \remmath{,}\addmath{.} \]
\rem{according to~\eqref{eqn:Jacobian}.}
From this it follows that 
\begin{align*}
\begin{split} 
& (1- Ch^2)\frac{2}{h} \int_{\M}\int_{B(0,1)\subseteq\T_x\M}\eta\left(|v|_x\right)\one_{A}(x)\one_{A^c}(\exp_x(h v)) \, \dd v \, \dd x \leq \TV_h(\one_A) \\
& \hspace{1cm} \leq (1+ Ch^2)\frac{2}{h} \int_{\M}\int_{B(0,1)\subseteq\T_x\M} \eta\left(|v|_x\right)\one_{A}(x)\one_{A^c}(\exp_x(h v)) \, \dd v \, \dd x
\end{split}
\end{align*}
which is the same as (using the notation from Lemma~\ref{lemm:subadditivity})
\begin{equation}
\label{eq:TVnonlocalBound}
(1- Ch^2) \frac{2g_A(h)}{h} \leq \TV_h(\one_A) \leq (1+ Ch^2) \frac{2g_A(h)}{h}.
\end{equation}

Let us now take $a$ such that $h \leq a < \frac{h_\M}{2}$, and let $N\in\N$ be such that
\[ (N-1) h < a \leq Nh. \]
Since $Nh \leq h_\M$ we have 
\[ \TV_{Nh}(\one_A) \leq \frac{(1+CN^2h^2) 2g_A(Nh)}{Nh} \leq \frac{(1+CN^2h^2) 2 g_A(h)}{h}, \]
by Lemma~\ref{lemm:subadditivity}. It now follows (using the lower bound for $\TV_h(\one_A)$ in \eqref{eq:TVnonlocalBound}) that
\[ \TV_{Nh}(\one_A) \leq \frac{(1+ C\addmath{N^2}h^2)}{\addmath{1-ch^2}}\TV_h(\one_A). \]
\rem{(where we absorb the factor of $N$ into the constant $C$).}
\add{By Taylor's theorem we have $\frac{1}{1-ch^2}=1+Ch^2+O(h^4)$ and so for $h$ sufficiently small we can assume $\frac{1}{1-ch^2}\leq 1+Ch^2$.
Moreover,
\[ Nh = (N-1)h + h \leq 2a. \]
So,
\[ \TV_{Nh}(\one_A) \leq (1+Ca^2)(1+ Ch^2)\TV_h(\one_A). \]}
Finally, notice that, from the 
choice of $N$, it follows that
\begin{align*}
\TV_{a}(\one_A) & \leq \left(\frac{Nh}{a}\right)^{m+1} \TV_{Nh}(\one_A) \\
& \leq \left(\frac{N}{N-1}\right)^{m+1} \TV_{Nh}(\one_A) \\
& \leq \left(\frac{N}{N-1}\right)^{m+1}\addmath{(1+Ca^2)}(1 + Ch^2) \TV_{h}(\one_A).
\end{align*}
Given the assumed smallness of $h$ and $a$ we can bound the right hand side by a constant times $TV_h(\one_A)$.

Now that we have proved the desired inequality for functions of the form $f=\one_A$ it is straightforward to extend it to general $f\in \Lp{1}(\M)$ by means of the coarea formula for the non-local total variation.
Indeed, using a layer cake representation, one can show that for every $f\in\Lp{1}(\M)$ and for every $h$ one has:
\[ \TV_h(f)= \int_{-\infty}^\infty \TV_{h}(\one_{\{ f \leq t \}}) \, \dd t; \]
\add{e.g.}\addii{, in Euclidean domains,~\cite{visintin1991generalized}.} 
From this it follows that
\[ \TV_{a}(f) = \int_{-\infty}^\infty \TV_{a}(\one_{\{ f \leq t \}}) \, \dd t \leq C \int_{-\infty}^\infty \TV_{h}(\one_{\{ f \leq t \}}) \, \dd t = C \TV_h(f). \qedhere \]
\end{proof}

In what follows we consider $\phi:[0,\infty) \rightarrow [0,\infty)$ a smooth function with compact support for which 
\[ \phi(t) \leq C \eta(t), \quad \forall t \geq0 \] 
for some $C>0$, and for which
\[ \int_{\R^m}\phi(|x|)\, \dd x =1. \]
We use $\phi_a$ to denote the rescaled version
\[ \phi_a(t):= \frac{1}{a^m} \phi \left(\frac{t}{a}\right), \]
and use it to define the \emph{smoothing operator}
\begin{align}
\begin{split}
&\Lambda_a : \Lp{1}(\M)  \rightarrow \Ck{\infty}(\M) \\
&\Lambda_a f(x)  = \frac{1}{\tau_a(x)}\int_{B_\M(x,a)}\phi_a(d_\M(x,y)) f(y) \, \dd\vol_\M(y), 
\label{def:Lambda}
\end{split}
\end{align}
where $\tau_a(x)$ is the normalization constant
\[ \tau_a(x) := \int_{B_\M(x,a)}\phi_a(d_\M(x,y)) \, \dd\vol_\M(y). \]
The parameter $a>0$ is a free parameter that we will choose later on.
 
Here we are attempting to mimic the construction of smoothing operators in \cite{BuragoIvanovKurylev} in the context of Dirichlet energies, i.e. the construction of an operator $\Lambda:\Lp{2}(\M) \rightarrow \Ck{\infty}(\M) $ for which a tight relationship between non-local and local Dirichlet energies can be obtained.
In our setting this amounts to finding an operator $\Lambda: \Lp{1}(\M) \rightarrow \Ck{\infty}(\M)$ satisfying, roughly speaking, 
\begin{equation}
\label{aux:Conv1}
\sigma_\eta \TV(\Lambda_a f) \leq (1+o(1))\TV_h(f) , \quad \forall f \in \Lp{\infty}(\M)
\end{equation}
as well as
\begin{equation}
\label{aux:Conv2}
\lVert \Lambda_a f - f \rVert_{\Lp{1}(\M)} \leq o(1)\TV_h(f)
\end{equation}
where $\sigma_\eta$ is the surface tension in~\eqref{def:sigmaeta}.
In~\cite{BuragoIvanovKurylev}, the analogous statement is obtained by selecting the smoothing operator as a convolution operator with respect to a conveniently chosen kernel and with a bandwidth of the same order as the connectivity lengthscale $h>0$.
The structure of the $\Lp{2}$-type Dirichlet seminorms makes the selection of this special kernel possible, but a kernel with similar properties does not seem to be found easily in the $\TV$ case.
As we will see below we will be forced to define our smoothing operator as a convolution operator whose kernel has a bandwidth $a>0$ that is much larger than the length-scale $h$ used to define the non-local TV seminorm, and more involved computations will be needed.
The subadditivity property for the TV seminorm, proved in Proposition~\ref{lemm:subadditivity}, combined with Proposition~\ref{prop:ConvolutionMonotone} below allow us to show that when $h \ll a$ we can still get the desired relations~\eqref{aux:Conv1} and~\eqref{aux:Conv2}, at the cost of losing some orders in the convergence rates.

We start by presenting some elementary properties of the smoothing operator $\Lambda_a$.

\begin{proposition}
\label{lem:Smoothening}
For every $0<h\leq a<\frac{h_\M}{2}$ and $k\in\N$ there exists $C$ (which may depend on $k$ and $\M$ but is independent of $a$ and $h$) such that
\begin{enumerate}
\item $\|\Lambda_af\|_{\Ck{k}(\M)} \leq Ca^{-k} \|f\|_{\Lp{\infty}(\M)}$ for all $f\in \Lp{\infty}(\M)$.
\item $\|f- \Lambda_a f\|_{\Lp{1}(\M)} \leq Ca \TV_h(f)$ for all $f \in \Lp{1}(\M)$.
\end{enumerate}
\end{proposition}

\begin{proof}
For the first inequality we just illustrate the case $k=1$, the other cases are obtained similarly (and anyway, are standard in the flat Euclidean case).
The gradient of $\Lambda_af$ is computed as
\begin{align*}
\nabla \Lambda_a f(x) & = - \frac{1}{\tau_a(x)a^{m+1}} \int_{\M} \phi'\left(\frac{d_\M(x,y)}{a} \right) \frac{\exp_x^{-1}(y)}{d_\M(x,y)} f(y) \, \dd\vol_\M(y) \\
& + \frac{\Lambda_a f(x)}{\tau_a(x) a^{m+1}} \int_{\M} \phi'\left(\frac{d_\M(x,y)}{a} \right) \frac{\exp_x^{-1}(y)}{d_\M(x,y)} \, \dd\vol_\M(y). 
\end{align*}
Taking the norm on both sides and using the triangle inequality we see that
\[ |\nabla \Lambda_a f(x)|_x \leq \frac{2\lVert f \rVert_{\Lp{\infty}(\M)}}{\tau_a(x)a^{m+1}} \int_{\M} \left|\phi'\left(\frac{d_\M(x,y)}{a} \right)\right| \, \dd\vol_\M(y). \]
This term on the other hand can be bounded by $Ca^{-1}\lVert f \rVert_{\Lp{\infty}(\M)}$.
This follows from the fact that $\tau_a(x)$ can be written as
\begin{equation}
\label{aux:lowerboundtaua}
\tau_a(x) =\frac{1}{a^m} \int_{B(0,a)}\phi\left(\frac{|v|_x}{a}\right)J_x(v) \, \dd v \geq(1-Ca^2)\frac{1}{a^m} \int_{B(0,a)}\phi\left(\frac{|v|_x}{a}\right) \, \dd v= (1-Ca^2), 
\end{equation}
and also
\begin{align*}
\frac{1}{a^m} \int_{\M} \left|\phi'\left(\frac{d_\M(x,y)}{a} \right)\right| \, \dd\vol_\M(y) & = \frac{1}{a^m}\int_{B(0,a)} \left|\phi'\left(\frac{|v|_x}{a}\right)\right|J_x(v) \, \dd v \\
 & \leq \frac{1+Ca^2}{a^m}\int_{B(0,a)}\left|\phi'\left(\frac{|v|_x}{a} \right)\right| \, \dd v \\
 & \leq C(1+Ca^2),
\end{align*}
thanks to the bounds on the Jacobian~\eqref{eqn:Jacobian}.

For the second identity we can write
\begin{align*}
\Lambda_a f (x) - f(x) & = \frac{1}{\tau_a(x)}\int_{\M}\phi_a(d_\M(x,y))f(y) \, \dd \vol_\M(y) - f(x) \\
& = \frac{1}{\tau_a(x)} \int_{\M}\phi_a(d_\M(x,y))(f(y)-f(x)) \, \dd \vol_\M(y),
\end{align*}
from where it follows that
\begin{align*}
\begin{split}
\lVert \Lambda_a f - f \rVert_{\Lp{1}(\M)} & \leq (1+ Ca^2) \int_{\M}\int_{\M}\phi_a(d_\M(x,y))|f(y)-f(x)| \, \dd \vol_\M(y) \, \dd \vol_\M(x) \\
& \leq \frac{C}{a^m}\int_{\M}\int_{\M}\eta\left(\frac{d_\M(x,y)}{a}\right)|f(y)-f(x)| \, \dd \vol_\M(y) \, \dd \vol_\M(x) \\
& = C a\TV_a(f) \leq Ca \TV_h(f),
\end{split}
\end{align*}
where for the first inequality we have used the lower bound on $\tau_a$ in \eqref{aux:lowerboundtaua}, for the second inequality we have used the fact that $\phi$ was chosen to be dominated by $\eta$, and in the last inequality we have used Proposition \ref{prop:monotonTV}.
\end{proof}

The next result is an important technical piece that we use in the sequel.
In the flat Euclidean case, the proof is quite elementary only involving a simple change of variables.
However, in the curved manifold setting, more involved computations are needed.
Three technical facts from differential geometry that are used in the proof are discussed in detail in the Appendix.

\begin{proposition}
\label{prop:ConvolutionMonotone}
(Monotonicity by convolution)
There are constants $C_1, C_2$ such that for all $0< h \leq a \leq \frac{h_\M}{2}$ and for all $f\in\Lp{\infty}(\M)$ we have
\[ \TV_{\widetilde{h}} (\Lambda_a f) \leq (1+ C_1a)\TV_h(f) + C_1 a \lVert f\rVert_{\Lp{\infty}(\M)}, \]
where $\widetilde{h}:=h(1- C_2a)$.
\end{proposition}

\begin{proof}
Let $\widetilde{h}:= h(1-Ca)$ for some constant $C$ that will be chosen later. We start by estimating $\TV_{\widetilde h}(\Lambda_a f)$, using the triangle inequality to bound it by the sum of:
\[A_1:= \frac{1}{\widetilde{h}^{m+1}} \int_{\M}\int_{\M} \left | \left(\frac{1}{\tau_a(x)}- \frac{1}{\tau_a(y)} \right)\widetilde \Lambda_af(x) \right | \eta\left( \frac{d_\M(x,y)}{\widetilde{h}}\right) \dd\vol_{\M}(x)\dd\vol_{\M}(y) \]
and
\begin{align*}
A_2:= \frac{1}{\widetilde{h}^{m+1}} \int_{\M}\int_{\M} \left | \frac{1}{\tau_a(y)}( \widetilde{\Lambda}_af(x) - \widetilde{\Lambda}_af(y)) \right | \eta\left( \frac{d_\M(x,y)}{\widetilde{h}}\right) \dd\vol_{\M}(x)\dd\vol_{\M}(y),
\end{align*}
where 
\[ \widetilde{\Lambda}_a f(x):= \int_{\M} \phi_a(d_\M(x,z))f(z)\dd\vol_\M(z).\]

Let us first bound the term $A_1$. Since we can bound $|\widetilde \Lambda_a f(x)|$ uniformly by $C \lVert f \rVert_{\Lp{\infty}(\M)}$ we just need to find a bound for 
\[ \frac{1}{\widetilde{h}^{m+1}} \int_{\M}\int_{\M} \left | \frac{1}{\tau_a(x)}- \frac{1}{\tau_a(y)} \right | \eta\left( \frac{d_\M(x,y)}{\widetilde{h}}\right) d\vol_{\M}(x)d\vol_{\M}(y) = \TV_{\widetilde{h}}( \tau_a^{-1}). \]
Now, using Proposition \ref{prop:biasineq} we know $\TV_{\widetilde{h}}(\tau_a^{-1})$ is smaller than $C\TV(\tau_a^{-1})$ (using also the smallness of $h$), which can be written as $C\TV(\tau_a^{-1})= C \int_{\M} |\nabla \tau_a^{-1}(x)| d\vol_\M(x) $ given that $\tau_a^{-1}$ is smooth (because $\phi$ is smooth and because $\tau_a$ is bounded away from zero as shown in \eqref{aux:lowerboundtaua}). Thus we just need to compute an estimate for $|\nabla \tau_a^{-1}(x)|$; this has already been done in \cite{BuragoIvanovKurylev}, but here we produce the argument for completeness. Indeed,
\[ \nabla \tau_a^{-1}(x) = -\frac{1}{\tau_a(x)^2} \nabla \tau_a(x) = \frac{1}{a^{m+1}\tau_a(x)^2} \int_{\M}\phi'\left(\frac{d_\M(x,z)}{a}\right) \frac{\exp_x^{-1}(z)}{d_\M(x,z)}\dd\vol_\M(z),\]
and in turn, the above integral can be written as
\[\int_{B(0,a)} \phi'(|v|_x/a)\frac{v}{|v|_x}J_x(v) \, \dd v = \int_{B(0,a)} \phi'(|v|_x/a)\frac{v}{|v|_x}(J_x(v)-1) \, \dd v, \]
where the last equality is due to radial symmetry. Using the estimates on the Jacobian \eqref{eqn:Jacobian} the norm of the above expression can be bounded by $Ca^{m+2}$ and thus $|\nabla \tau_a^{-1}(x)|\leq C a$. Putting all estimates together we conclude that $A_1 \leq C\lVert f \rVert_{\Lp{\infty}(\M)} a$.

Now we focus on estimating $A_2$. First of all, given the lower bounds for $\tau_a$ in \eqref{aux:lowerboundtaua} we can focus on finding bounds for $\TV_{\widetilde{h}}(\widetilde{\Lambda}_a f)$. In other words, thanks to \eqref{aux:lowerboundtaua} we already have
\[ A_2 \leq (1+Ca^2) \TV_{\widetilde{h}}(\widetilde {\Lambda}_a f ), \]
and so, we just need to bound $\TV_{\widetilde{h}}(\widetilde {\Lambda}_a f )$. To achieve this, we start by noticing that for any given $x,y\in \M$ satisfying $d_\M(x,y)\leq \widetilde{h}$ we can write
\begin{align*}
& \widetilde{\Lambda}_af(x) =\frac{1}{a^m} \int_{B(0,a) \subseteq \T_x\M}\phi(|v|_x/a)f(\exp_x(v))J_x(v)\,\dd v, 
\\& \widetilde{\Lambda}_af(y) =\frac{1}{a^m} \int_{B(0,a) \subseteq \T_y\M}\phi(|v'|_y/a)f(\exp_y(v'))J_y(v')\, \dd v'.  
\end{align*}
We attempt to write this last expression in terms of an integral over tangent vectors at $x$. For that purpose we consider $PT_{x,y}: \T_x\M \rightarrow \T_y\M$ the parallel transport from $x$ to $y$ along a constant speed geodesic connecting $x$ and $y$. By definition of parallel transport, the map $PT_{x,y}$ is an isometry between $\T_x\M$ and $\T_y \M$ and hence its Jacobian is equal to one. Therefore, we can write 
\[ \widetilde{\Lambda}_af(y) =\frac{1}{a^m} \int_{B(0,a) \subseteq \T_x\M}\phi(|v|_x/a)f(\exp_y(PT_{x,y}(v)))J_y(PT_{x,y}(v))\, \dd v. \]
The above expression allows us to write the difference $\widetilde \Lambda_af(x)-\widetilde \Lambda_af(y)$ in terms of a single integral
\[ \frac{1}{a^m} \int_{B(0,a) \subseteq \T_x\M}\phi(|v|_x/a)\left(f(\exp_x(v))J_x(v)-f(\exp_y(PT_{x,y}(v)))J_y(PT_{x,y}(v))\right)\, \dd v. \]
Now, as we show in the Appendix, there is a constant $C$ such that
\begin{equation}
|J_x(v) - J_y(PT_{x,y}(v)) | \leq C ah, 
\label{app:JacobiansEstiamtes}
\end{equation}
for all $x \in \M$, all $y \in B_\M(x,h)$ and all $v\in T_x \M$ with $|v|\leq a$. As a consequence, we can use the triangle inequality to bound $\TV_{\widetilde{h}}(\widetilde{\Lambda}_af)$ by the sum of:
\[ A_{2,1}:= C a\lVert f \rVert_{\Lp{\infty}(\M)} \]
and
\[ A_{2,2}:=\frac{1}{\widetilde{h}}\int_{\M}\int_{\M}\int_{\T_x \M} \eta_{\widetilde{h}}\left( d_\M(x,y)\right)\phi_a(|v|_x)| F(x, \exp_{x}^{-1}(y),v) | J_x(v) \, \dd v \, \dd \vol_{\M}(y) \, \dd \vol_\M(x),\]
where
\[ F(x,w,v):= f(\exp_x(v)) - f (\exp_{\exp_x(w)}(PT_{x, \exp_x(w)}(v) )). \]
%
%
%
$A_{2,2}$ can be rewritten as
\begin{align*}
&A_{2,2}= \frac{1}{\tilde h} \int_\M \int_{\T_x \M} \int_{\T_x \M} \eta_{\tilde h} \left( |w|_x \right) \phi_a(|v|_x) |F(x,w,v)| J_x(w) J_x(v) \, \dd v \, \dd w \, \dd \vol_\M(x)
\\& \leq \frac{(1+ Ca^2+ Ch^2)}{\tilde h} \int_\M \int_{\T_x \M} \int_{\T_x \M} \eta_{\tilde h} \left( |w|_x \right) \phi_a(|v|_x) |F(x,w,v)| \, \dd v \, \dd w \, \dd \vol_\M(x)
\end{align*}
where the inequality is due to the upper bound on $J_x$ from \eqref{eqn:Jacobian}. It will be convenient to rewrite the term on the right hand side of the inequality as a single integral
\begin{equation}
\frac{(1+ Ca^2 + Ch^2)}{\tilde h} \int_{\T^2 \M} \eta_{\tilde h} \left( |w|_x \right) \phi_a(|v|_x) |F(x,w,v)| \, \dd\vol_{\T^2 \M}(x,w,v),  
\label{aux:1}
\end{equation}
over the fiber bundle 
\[ \mathcal{T}^2\M:= \{ (x,v_1,v_2) \: : \: x \in \M, \quad v_1, v_2 \in \T_x\M \},\]
whose volume form $\vol_{\T^2\M}$ can be understood by $\dd\vol_{\T^2\M}(x,v_1,v_2) = \dd v_1 \, \dd v_2 \, \dd \vol_\M(x)$ (and we refer to the appendix for a rigorous formulation of the volume form on $\T^2\M$). 
We would also like to notice that the integrand in \eqref{aux:1} is zero outside the set
\[ \mathcal{B}^2_{\widetilde{h},a}:= \{ (x,v_1,v_2) \in \mathcal{T}^2 \M \: :\: |v_1|_x\leq \tilde h ,\quad |v_2|_x \leq a \}. \]

Next, we define a transformation $\widetilde{\Psi}$ of $(x,v_1,v_2)$ which will allow us to simplify the expression in \eqref{aux:1}. We let
\[\widetilde{\Psi}: (x,w,v) \longmapsto (\tilde x , \tilde w , \tilde v) \] 
be the map
\begin{equation}
\tilde x:= \Phi_1^1(x,v), \quad \tilde v := \Phi_1^2(x,v), \quad \tilde w:= \exp_{\tilde x}^{-1}\left( \exp_{\exp_x(w)}(PT_{x, \exp_x(w)}(v) ) \right),  
\label{def:tildes} 
\end{equation}
defined for all $(x,w,v) \in \mathcal{B}^2_{c_1, c_2}$; here, $c_1, c_2$ are order one quantities that are sufficiently small so as to guarantee that $\widetilde{\Psi}$ is a diffeomorphism. It is straightforward to verify that the image of $\widetilde{\Psi}$ is contained in $\mathcal{B}^2_{\tilde{c}_1, c_2}$ for some order one quantity $\tilde c_1$. In the Appendix we show that
\begin{equation}
\left | |w|_x -|\tilde w|_{\tilde x} \right| \leq Cah  
\label{app:DistanceEstimate}
\end{equation} 
for all $(x,w,v) \in \mathcal{B}^2_{ \widetilde{h},a}$, and in particular, for such a triple we have $|\tilde w|_{\tilde x} \leq \tilde h + Cah = h $ (here is where we make the choice of $C$ in the definition of $\tilde h$, i.e. $\tilde h: = h(1-Ca) $).
The bottom line is that for all $(x,v,w) \in \T^2 \M$ we have
\[ \eta\left( \frac{|w|_x}{\tilde h} \right) \phi_a(|v|_x ) \leq \eta\left( \frac{|\tilde w|_{\tilde x}}{h} \right) \phi_a(|v|_x ). \]
Furthermore, from the definition of $\tilde v$ it follows that
\[ |v|_{x} = |\tilde v|_{\tilde x}. \]
Thus, for all $(x,v,w)$ we have
\[ \eta\left( \frac{|w|_x}{\tilde h} \right) \phi_a(|v|_x ) \leq \eta\left( \frac{|\tilde w|_{\tilde x}}{h} \right) \phi_a(|\tilde v|_{\tilde x} ). \]
Now, from the above inequality and the fact that
\[ F(x,w,v)= f(\tilde x)- f( \exp_{\tilde x}(\tilde w) ) \]
we can then upper bound \eqref{aux:1} by
\[ ( 1+Ca) \frac{1}{h} \int_{\T^2 \M} \eta_h(|\tilde w|_{\tilde x})| \phi_a(|\tilde v |_{\tilde x})| f(\tilde x) - f(\exp_{\tilde x} (\tilde w)) | \, \dd\vol_{\T^2 \M}(x,w,v), \]
where we have also used the smallness of $a$ to expand $(1-Ca)^{m+1}$, and have used the fact that $h \leq a$. We can then change variables to rewrite the above expression as
\begin{equation}
( 1+Ca) \frac{1}{h} \int_{\T^2 \M} \eta_h(|\tilde w|_{\tilde x})| \phi_a(|\tilde v |_{\tilde x})|f(\tilde x) - f(\exp_{\tilde x} (\tilde w)) | \left| \frac{\partial \widetilde{\Psi}^{-1}(\tilde{x},\tilde{w},\tilde{v}) }{\partial (\tilde{x},\tilde{w},\tilde{v})} \right|\,\dd\vol_{\T^2 \M}(\tilde{x},\tilde{w},\tilde{v}),  
\label{aux:2}
\end{equation}
where $\left| \frac{\partial \widetilde{\Psi}^{-1}(\tilde{x},\tilde{w},\tilde{v}) }{\partial (\tilde{x},\tilde{w},\tilde{v})} \right|$ is the Jacobian of the transformation $\widetilde {\Psi}^{-1}$. As we prove in the Appendix, the Jacobian satisfies
\begin{equation}
\left | \left| \frac{\partial \widetilde{\Psi}^{-1}(\tilde{x},\tilde{w},\tilde{v}) }{\partial (\tilde{x},\tilde{w},\remmath{\tilde{w}}\addmath{\tilde{v}})} \right| - 1 \right | \leq Ca
\label{app:JacobiansPsi}
\end{equation}
for all $(\tilde{x},\tilde{w},\tilde{v}) \in \widetilde{\Psi}(\mathcal{B}_{\tilde{h},a}^2)$. 

From \eqref{app:JacobiansPsi} it follows that the term in \eqref{aux:2} can be bounded by
\[ ( 1+Ca) \frac{1}{h} \int_{\T^2 \M} \eta_h(|\tilde{w}|_{\tilde{x}})| \phi_a(|\tilde{v} |_{\tilde{x}})| f(\tilde{x}) - f( \exp_{\tilde{x}}(\tilde{w})) | \, \dd\vol_{\T^2 \M}(\tilde{x},\tilde{w},\tilde{v}). \]
This expression can be written as,
\[ (1+Ca) \frac{1}{h} \int_{\M}\int_{\T_{\tilde{x}}\M} \int_{\T_{\tilde{x}}\M} \eta_h(|\tilde{w}|_{\tilde{x}})| \phi_a(|\tilde{v}|_{\tilde{x}})| f(\tilde{x}) - f( \exp_{\tilde{x}}(\tilde{w})) | \, \dd \tilde{v}\, \dd \tilde{w} \, \dd \vol_\M(\tilde{x}), \]
multiplying and dividing its integrand by $J_{\tilde{x}}(\tilde{w})$, using the bounds on the Jacobian in \eqref{eqn:Jacobian} and the upper bound $\tau_a(\tilde{x})\leq1+Ca^2$ (derived analogously to~\eqref{aux:lowerboundtaua}), we can deduce that the above expression is smaller than 
\[ (1+Ca) \frac{1}{h} \int_{\M}\int_{\T_{\tilde{x}}\M} \int_{\T_{\tilde{x}}\M} \eta_h(|\tilde{w}|_{\tilde{x}})| \phi_a(|\tilde{v}|_{\tilde{x}})| f(\tilde{x}) - f( \exp_{\tilde{x}}(\tilde{w})) |J_{\tilde{x}}(\tilde{w}) \, \dd \tilde{v} \, \dd \tilde{w} \, \dd\vol_\M(\tilde{x}). \]
Upon integration on $\tilde{v}$, we recognize that the above term is precisely
\[ (1+Ca) \TV_h(f). \]
Putting all the estimates together, we deduce that: 
\[ \TV_{\widetilde h} (\Lambda_a f) \leq (1+ Ca)\TV_h(f) + C a \lVert f\rVert_{\Lp{\infty}(\M)} \]
as we wanted to show.
\end{proof}

Combining all propositions above we obtain the following.

\begin{corollary}\label{cor:smoothening-summary}
There exists a constant $C>0$, such that for all $0< h\leq a \leq \frac{h_\M}{2}$, we have for all $f \in \Lp{\infty}(\M)$,
\begin{enumerate}
\item $\sigma_\eta \TV(\Lambda_a f) \leq (1+ C(h^2 + a))\TV_h(f) + C \left(\frac{h}{a^2} + a \right) \lVert f \rVert_{\Lp{\infty}(\M)}$;
\item $\lVert \Lambda_a f - f \rVert_{\Lp{1}(\M)} \leq C a \TV_h(f)$.
\end{enumerate}
\end{corollary}

\begin{proof}
From Proposition \ref{lem:Smoothening} we know that $\Lambda_a f \in \Ck{2}(\M)$ for every $f\in \Lp{\infty}(\M)$ and moreover
\[\lVert \Lambda_a f \rVert_{\Ck{2}(\M)} \leq \frac{C}{a^2}\lVert f \rVert_{\Lp{\infty}(\M)}. \]
Thus, by Proposition \ref{prop:NonLocalLocalRegu}, it follows that
\[ \sigma_\eta \TV(\Lambda_a f) \leq (1+ C\widetilde{h}^2) \TV_{\widetilde{h}}(\Lambda_a f) + \frac{C h}{a^2} \lVert f \rVert_{\Lp{\infty}(\M)}, \]
where $\widetilde{h}= h(1-Ca)$.
Using Proposition \ref{prop:ConvolutionMonotone} we deduce 
\[ \sigma_\eta \TV(\Lambda_a f) \leq (1+ C(h^2 + a))\TV_h(f) + C \left(\frac{h}{a^2} + a \right) \lVert f \rVert_{\Lp{\infty}(\M)} \]
as we wanted to show.
The second inequality is just ii) in Proposition~\ref{lem:Smoothening}.
\end{proof}

In the remainder we make use of the following two lemmas.

\begin{lemma}
\label{Lemma:SmallnessMini}
For every $C_0>0$ there exists $\beta_0>0$ and $h_\M$ such that for every $0<h<h_\M$, and for every $A\subseteq \M$ for which
\[ \frac{\TV_{h}(\one_A)}{\min \{ \vol_\M(A),\vol_\M(A^c)\}}\leq C_0, \]
we have
\[ \min \{ \vol_\M(A),\vol_\M(A^c) \} \geq \beta_0. \]
\end{lemma}

\begin{proof}
Let us fix $C_0>0$.
For the sake of contradiction suppose the result is not true.
In that case we would be able to build a sequence $\{h_k \}_{k \in \N}$ with $h_k \rightarrow 0$ as $k \rightarrow \infty$ and a sequence of sets $\{A_k \}_{k \in \N}$ satisfying
\[ \frac{\TV_{h_k}(\one_{A_k})}{\min \{ \vol_\M(A_k),\vol_\M(A_k^c) \}}\leq C_0, \quad \min \{ \vol_\M(A_k),\vol_\M(A_k^c) \} \leq \frac{1}{k}. \]
Without loss of generality we can assume $\min \{ \vol_\M(A_k),\vol_\M(A_k^c) \}= \vol_\M(A_k)$ (for otherwise we can take complements).
Let us now introduce rescaled functions
\[ \widetilde{f}_k := \alpha_k \one_{A_k}, \quad \alpha_k := \frac{1}{\vol_\M(A_k)}, \quad k \in \N. \]
This sequence satisfies:
\begin{enumerate}
\item $\lVert \widetilde{f}_k \rVert_{\Lp{1}(\M)}=1$, for all $k \in \N$,
\item $\TV_{h_k}(\widetilde{f}_k)\leq C_0$, for all $k \in \N$. 
\end{enumerate}
Thanks to the above properties we can use the compactness result from~\cite[Lemma 4.4]{GarciaTrillos2015} to conclude that there exists $\widetilde{f} \in \Lp{1}(\M)$ such that (up to subsequence not relabeled) 
\[ \widetilde{f}_k \rightarrow_{\Lp{1}(\M)} \widetilde{f}. \]
We notice that although the result in~\cite{GarciaTrillos2015} is stated in the flat Euclidean case, it is still good enough for our purpose as one can work in local coordinates and use the compactness and smoothness of $\M$ and a gluing argument to extend the local result to a global one. 

Now, observe that the function $\widetilde{f}$ must satisfy $\lVert \widetilde{f} \rVert_{\Lp{1}(\M)}=1$.
Moreover, since the functions $\widetilde{f}_{k}$ are of the form $\widetilde{f}_k = \alpha_k \one_{A_k}$, necessarily the function $\widetilde{f}$ has the form $\widetilde{f}= \alpha \one_{A}$ for a positive real number $\alpha$.
It should also hold that
\[ \vol_\M(A_k) \rightarrow \vol_\M(A), \quad k \rightarrow \infty. \]
From this and the fact that $\alpha \vol_\M(A) =1$ we conclude that
\[ \alpha = \frac{1}{\vol_\M(A)}= \lim_{k \rightarrow \infty} \frac{1}{\vol_\M(A_k)}. \]
This however contradicts the fact that $\frac{1}{\vol_\M(A_k)} \geq k$ for all $k \in \N$.
\end{proof}

The next lemma states that the Cheeger constant $\CM$ can also be written as the minimum value of an optimization problem over BV functions taking values on $[0,1]$.
The proof is standard and we present it for the convenience of the reader. We refer the reader to Section 2 in \cite{Chung1997} for a proof of an analogous result in the graph setting.

\begin{lemma}
\label{lem:CheegerFunctions}
The Cheeger constant $\CM$ admits the representation
\[ \CM = \min_{f:\M \rightarrow [0,1]} \frac{\TV(f)}{\lVert f- m_1(f)\rVert_{\Lp{1}(\M)}} \]
where in the above, $m_1(f)$ is a median of $f$, i.e., any number $c$ in $[0,1]$ where the minimum
\begin{equation} \label{eq:median}
\min_{c\in [0,1]} \lVert f -c \rVert_{\Lp{1}(\M)}
\end{equation}
is achieved.
\end{lemma}
\begin{proof}
A simple computation shows that for a BV set $E \subseteq \M $ one has
\[  \lVert \mathds{1}_{E} -m_1(\mathds{1}_E) \rVert_{\Lp{1}(\M)} = \min \{ \vol_\M(E), \vol_\M(E^c) \} \]
and thus is clear that
\[ \mathcal{C}_\M \geq \min_{f:\M \rightarrow [0,1]} \frac{\TV(f)}{\lVert f- m_1(f)\rVert_{\Lp{1}(\M)}}.  \]
Conversely, for a BV function $f: \M \rightarrow [0,1]$ 
we let $g:= f -m_1(f)$. We claim
\begin{align*}
\vol_\M(\{g\leq t\}) & \leq \vol_\M(\{g>t\}) && \hspace{-3cm} \forall t\leq 0 \\
\vol_\M(\{g\leq t\}) & \geq \vol_\M(\{g>t\}) && \hspace{-3cm} \forall t\geq 0.
\end{align*}
Since the LHS is increasing and the RHS is decreasing then it is enough to show that
\[ \vol_\M(\{g\leq 0\}) = \vol_\M(\{g>0\}). \]
This is just the optimality condition for~\eqref{eq:median} so it holds by definition of $m_1(f)$.
Now,
\begin{align*}
 \begin{split}
 \TV(f)= \TV(g) & = \int_{-\infty}^{\infty}\TV( \mathds{1}_{ \{ g\leq t \}}) \, \dd t \\
 & \geq \mathcal{C}_\M\left( \int_{-\infty}^0 \vol_{\M}( \{ g \leq t  \}) \dd t + \int_{0}^\infty \vol_{\M}( \{ g > t  \}) \, \dd t   \right) \\
 & = \mathcal{C}_\M \lVert g\rVert_{\Lp{1}(\M)}.
 \end{split}
\end{align*}
Hence,
\[ \frac{\TV(f)}{\lVert f- m_1(f) \rVert_{\Lp{1}(\M)}} \geq \mathcal{C}_\M. \] 	
From the fact that $f$ was arbitrary we obtain the desired inequality.
\end{proof}

\section{Proofs of the Main Results}
\label{sec:ProofMainResults}

\add{Sections~\ref{sec:Upper} and~\ref{sec:Interpolation} combine to prove Theorem~\ref{thm:cheeger-constant-conv}. In Section~\ref{sec:CheegerCutsProof} we prove Theorem~\ref{thm:rates-Cheeger-Cuts}.}

\subsection{Upper Bound}
\label{sec:Upper}

In this section we quantify the relationship
\[  \mathcal{C}_{n,\veps} \lesssim  \mathcal{\sigma_\eta}\CM. \]
For this purpose we discuss pointwise estimates for the approximation of $\TV(f)$ with $\GTV_{n,\e}(f)$ for a \textit{fixed} BV function $f : \M \rightarrow \R$. We begin by stating an estimate for $|\GTV_{n,\e}(f)- \E(\GTV_{n,\e}(f))|$.
\begin{proposition}
\label{prop:VarianceEstimate}
\add{There is a constant $C>0$ such that for all $0<\veps<\veps_0$ (i.e. $\veps$ small enough), all $0<\zeta < \zeta_0$ (all $\zeta$ small enough) and all $f\in\BV(\M)\cap\Lp{\infty}(\M)$ satisfying $n\zeta\eps^{\frac{m+1}{2}}\geq \l\|f\|_{\Lp{\infty}(\M)}+\TV(f)\r \geq \sqrt{\zeta}$, we have: 
\[
  \Prob\left( \left| \frac{n}{n-1}\GTV_{n,\e}(f) - \widetilde{\TV}_\veps(f)  \right|   \geq   ( 1+ \frac{1}{R(f)} ) \zeta \right) \leq C \exp\left(-\frac{Cn{\zeta} \min\{\e^{\frac{m+1}{2}},\eps\zeta\}}{(R(f))^2}\right),
\]}
where
\begin{equation*}
\widetilde{\TV}_{\veps}(f):= \frac{1}{\veps^{m+1}}\int_\M \int_\M |f(x) - f(y)| \eta\left( \frac{|x-y|}{\veps} \right)\dd \vol_\M(x) \dd \vol_\M(y), \quad f \in \Lp{1}(\M), 
\end{equation*}
and \add{
\[ R(f):= \|f\|_{\Lp{\infty}(\M)} + \TV(f).  \]
}
In addition, with probability at least  $1-C \exp\left(-\frac{Cn {\zeta} \min\{ \e^{\frac{m+1}{2}},\eps\zeta\}}{(R(f))^2}\right)$ we have that \add{
\[
  \GTV_{n,\e}(f) \leq \sigma_\eta(1+C\e^2) \TV(f) + (1+R(f)) {\zeta}.
\]
}
\end{proposition}

\begin{proof}
When $f$ is an indicator function, these types of estimates have been previously proved in \cite{trillos2017estimating}. We follow the same proof, highlighting changes that occur due to considering general BV functions.

Since the variables $x_1, \dots, x_n$ are i.i.d. samples, $\GTV_{n,\e}(f)$ is a \emph{$U$-statistic} \cite{GineUStats} of order two as it can be written as
\[
  \GTV_{n,\e}(f) = \frac{1}{n^2}\sum_{i\neq j} \phi(x_i,x_j),\qquad \phi(x,y):= \frac{1}{\veps^{m+1}}\eta\left( \frac{|x-y|}{\e}\right) |f(x)-f(y)|.
\]
Notice that \[ \E(\GTV_{n,\e}(f)) = \frac{n-1}{n} \widetilde{\TV_{\veps}}(f).\]
We remark that the only difference between $\widetilde{\TV}_\veps(f)$ and $\TV_\veps(f)$ defined in \eqref{eqn:NonLocalTvseminorm} is the fact that in the former we use the Euclidean distance to determine the proximity of points whereas for the latter we use the geodesic distance. We will first provide concentration inequalities for $\widetilde{\TV}_\e(f)$, and then relate back to $\TV_\veps(f)$.

Using the Hoeffding decomposition of U-statistics, we will write
\[ \frac{n}{n-1}\GTV_{n,\e}(f) - \widetilde{\TV}_\veps(f)= 2U_{1} + U_2  \]
where 
\[ U_1:= \frac{1}{n}\sum_{i=1}^n g_{1}(x_i) \]
and
\[ U_2:= \frac{2}{n(n-1)}\sum_{i=1}^n \sum_{j>i}g_{2}(x_i, x_j). \]
$U_1$ and $U_2$ are \textit{canonical} $U$-statistics of order one and two respectively. In the above,
\[ g_1(x):= \overline{\phi}(x) -\widetilde{\TV}_\veps(f) \]
\[ g_2(x,y):= \phi(x,y) - \overline{\phi}(x) - \overline{\phi}(y)+ \widetilde{\TV}_\veps(f) \]
\[ \overline{\phi}(x):= \int_{\M}\phi(x,y) \, \dd \vol_{\M}(y).\]
We can then make use of the concentration inequalities for canonical $U$ statistics discussed in section 3 of \cite{GineUStats}, which imply that: for all $t>0$,
\[ \Prob\left(|U_1| \geq \frac{t}{n} \right) \leq K \exp\left(\frac{- t^2}{K(t A_1 + B_1^2)} \right), \]
and, for all $t>\frac{B_2}{K}$,
\[ \Prob\left(|U_2| \geq \frac{t}{n(n-1)} \right) \leq K \exp\left(-\frac{1}{K} \min \left \{ \left(\frac{t}{A_2} \right)^{1/2}, \frac{t}{B_2}, \left( \frac{t}{C_2} \right)^{2/3}   \right \} \right),  \]
for a universal constant $K>0$. In the above,
\[A_1:= \lVert g_1 \rVert_{\Lp{\infty}(\M)}, \quad B_1:= \sqrt{n}\lVert g_1 \rVert_{\Lp{2}(\M)}, \]
and
\[A_2:= \lVert g_2 \rVert_{\Lp{\infty}(\M\times\M)}, \quad B_2:= n \lVert g_2 \rVert_{\Lp{2}(\M \times \M)},  \quad (C_2)^2:=n \left \lVert \int_{\M}g_2^2(\cdot,y) \, \dd \vol_\M(y) \right \rVert_{\Lp{\infty}(\M)}. \]
We remark that the first inequality controls the order one U-statistic using precisely Bernstein's inequality.

We seek to control all the above quantities using estimates as in \cite{trillos2017estimating}, with the modification that we allow $f$ to be a function of bounded variation. For the $\Lp{\infty}$ estimates, by recalling the definition of $\phi,\bar \phi$, \add{we see that}
\[
 A_1 \leq \frac{C \lVert f \rVert_{\Lp{\infty}(\M)}}{\veps},\qquad  A_2 \leq \frac{C\lVert f \rVert_{\Lp{\infty}(\M)}}{\veps^{m+1}}.
\]
On the other hand, in estimating the squared terms, one needs to estimate various integrals of $\phi^2$ and $\bar \phi^2$. In the case of $B_2$, after some straightforward expansions we obtain
\[
  B_2^2 = n^2 \left(\int_\M \int_\M \phi(x,y)^2 \,\dd\vol_\M(x) \,\dd\vol_\M(y) + (\widetilde{\TV}_\e(f))^2 -2 \int_\M \bar \phi^2(x)\, \dd\vol_\M(x)\right).
\]
Using Jensen's inequality, and then bounding $\phi^2 \leq \frac{2\|f\|_{\Lp{\infty}(\M)}}{\e^{m+1}}\phi$, we can write, for $\eps\leq \eps_0$,
\[
  B_2^2 \leq Cn^2\left(\frac{\|f\|_{\Lp{\infty}(\M)}\widetilde{\TV}_\e(f)}{\e^{m+1}} + (\widetilde{\TV}_\e(f))^2\right) \leq Cn^2 \frac{(\|f\|_{\Lp{\infty}(\M)} +\widetilde{\TV}_\e(f))^2 }{\e^{m+1}}
\]
Similarly, by expanding $g_2$, we can bound, for $\eps\leq \eps_0$,
\begin{align*}
  C_2^2 &\leq 2n \left( \left\|\int_\M \phi(\cdot,y)^2 \, \dd\vol_\M(y)\right\|_{\Lp{\infty}(\M)} \hspace{-8pt} +  \left\|\int_\M \bar\phi(y)^2\,\dd\vol_\M(y)\right\|_{\Lp{\infty}(\M)} \hspace{-8pt} +  \left\| \bar\phi^2\right\|_{\Lp{\infty}(\M)} + \widetilde{\TV}_\e(f)^2 \right)\\
  &\leq Cn \left( \frac{\|f\|_{\Lp{\infty}(\M)}^2}{\e^{m+2}} + \frac{\|f\|_{\Lp{\infty}(\M)}\widetilde{\TV}_\e(f)}{\e^{m+1}} + \widetilde{\TV}_\e(f)^2 \right) \leq Cn \frac{(\|f\|_{\Lp{\infty}(\M)} +\widetilde{\TV}_\e(f))^2 }{\e^{m+2}}.
\end{align*}
To bound $B_1$\rem{, by using the same change of variables argument as in Lemma 2.1 in} \cite{trillos2017estimating}, \rem{but inserting the appropriate $\Lp{\infty}$ bound on $f$, one may show that} \add{we first compute}
\begin{align*}
\addmath{\int_{\M} \bar{\phi}^2(x) \, \dd\vol_{\M}(x)} & \addmath{= \frac{1}{\eps^{2(m+1)}} \int_{\M} \l \int_{\M} \eta\l\frac{|x-y|}{\eps}\r |f(x) - f(y)| \, \dd \vol_{\M}(y) \r^2 \, \dd \vol_{\M}(x) }\\
 & \addmath{ \leq \frac{1}{\eps^{2(m+1)}} \int_{\M} \l \int_{\M} \eta\l\frac{d_{\M}(x,y)}{C\eps}\r |f(x) - f(y)| \, \dd \vol_{\M}(y) \r^2 \, \dd \vol_{\M}(x) }\\
 & \addmath{ \hspace{4cm} \text{since } C|x-y|\geq d_{\M}(x,y) }\\
 & \addmath{ = \frac{C}{\eps^2} \int_{\M} \l \int_{B(0,1)\subseteq \T_x\M} \eta(|w|_x) |f(x)-f(\exp_x(C\eps w))| J_x(C\eps w) \, \dd w \r^2 \, \dd \vol_{\M}(x). }
\end{align*}
\add{Now, since}
\begin{align*}
& \addmath{\int_{B(0,1)\subseteq \T_x\M} \eta(|w|_x) |f(x)-f(\exp_x(C\eps w))| J_x(C\eps w) \, \dd w} \\
& \hspace{1cm} \addmath{ \leq 2\|f\|_{\Lp{\infty}(\M)} \int_{B(0,1)\subseteq \T_x\M} \eta(|w|_x) (1+C|\eps w|^2) \, \dd w }\\
& \hspace{1cm} \addmath{ \leq C\|f\|_{\Lp{\infty}(\M)} },
\end{align*}
\add{then}
\begin{align*}
\addmath{\int_{\M} \bar{\phi}^2(x) \, \dd\vol_{\M}(x)} & \addmath{\leq \frac{C\|f\|_{\Lp{\infty}(\M)}}{\eps^2} \int_{\M} \int_{B(0,1)\subseteq \T_x\M} \eta(|w|_x) |f(x)-f(\exp_x(C\eps w))| J_x(C\eps w) \, \dd w \, \dd \vol_{\M}(x) }\\
 & \addmath{= \frac{C\|f\|_{\Lp{\infty}(\M)}}{\eps^{2+m}} \int_{\M} \int_{\M} \eta\l\frac{d_{\M}(x,y)}{C\eps}\r |f(x)-f(y))| \, \dd \vol_{\M}(y) \, \dd \vol_{\M}(x) }\\
 & \addmath{= \frac{C\|f\|_{\Lp{\infty}(\M)}}{\eps} \TV_{C\eps}(f) }\\
 & \addmath{\leq \frac{C\|f\|_{\Lp{\infty}(\M)}}{\eps} \TV(f) }
\end{align*}
\add{by Proposition~\ref{prop:biasineq}.}
Hence expanding again we obtain
\[ B_1^2 \leq \frac{Cn}{\e}\left(\TV(f)\|f\|_{\Lp{\infty}(\M)} + \widetilde{\TV}_\e(f)^2\right). \]
Now we provide a way to compare $\TV(f)$ and $\widetilde{\TV}_\e(f)$.
First, we recall (see~\cite[Proposition 2]{garciatrillos19}) that 
\[ d_\M(x,y) \leq |x-y| + C|x-y|^3. \]
In turn, after applying Proposition~\ref{prop:biasineq} we obtain
\begin{equation}
\widetilde{\TV}_{\veps}(f)\leq \left(1+C\veps^2 \right)^{m+1} \TV_{\veps(1+ C\veps^2)}(f) \leq (1+ C\veps^2)\TV_{\veps(1+ C\veps^2)}(f) \leq (1+C \e^2)\sigma_\eta \TV(f), 
\label{eqn:TVEuclideanGeodesic}
\end{equation}
\add{where in the above the constant $C$ changes from inequality to inequality.} We obtain:
\[ B_1^2 \leq \frac{Cn}{\eps} \l\TV(f)\|f\|_{\Lp{\infty}(\M)} +\TV(f)^2\r \leq\frac{Cn}{\eps}\l\TV(f)+\|f\|_{\Lp{\infty}(\M)}\r^2, \]
and
\begin{align*}
B_2^2 & \leq \frac{Cn^2}{\eps^{m+1}}\l\|f\|_{\Lp{\infty}(\M)}+\TV(f)\r^2 \\
C_2^2 & \leq \frac{Cn}{\eps^{m+2}}\l\|f\|_{\Lp{\infty}(\M)}+\TV(f)\r^2.
\end{align*}

Hence, for $t>0$, we find that
\begin{equation} \label{eq:U1}
\Prob\left(|U_1| \geq \frac{t}{n} \right) \leq C\exp\left(\frac{- Ct^2}{t \frac{\|f\|_{\Lp{\infty}(\M)}}{\e} + n  \frac{(\|f\|_{\Lp{\infty}(\M)} +\TV(f))^2 }{\e}} \right),
\end{equation}
and, for $t\geq \frac{Cn}{\eps^{\frac{m+1}{2}}}\l\|f\|_{\Lp{\infty}(\M)}+\TV(f)\r$,
\begin{equation} \label{eq:U2}
\Prob\left(|U_2| \geq \frac{t}{n(n-1)} \right) \leq C e^{-CF_{t,\eps,n}(f)},
\end{equation}
where
\[ F_{t,\eps,n}(f) = \min\lb\l\frac{t\e^{m+1}}{\|f\|_{\Lp{\infty}(\M)}} \r^{1/2}, \frac{t\e^{\frac{m+1}{2}}}{n(\|f\|_{\Lp{\infty}(\M)} +\TV(f))}, \l\frac{t\e^{\frac{m+2}{2}}}{ (\|f\|_{\Lp{\infty}(\M)} +\TV(f)) \sqrt{n}} \r^{2/3} \rb. \]
Choosing $t=\frac{n(n-1)\zeta}{2\l \|f\|_{\Lp{\infty}}+\TV(f)\r}$ in \eqref{eq:U2}, and from the upper bound $\sqrt{\zeta}\leq\|f\|_{\Lp{\infty}(\M)}+\TV(f)$ \add{(which is the same as $ \zeta^{1/3}\leq\l\|f\|_{\Lp{\infty}(\M)}+\TV(f)\r^{2/3}$)}, we can infer that 
\[ F_{t,\eps,n}(f) \geq \frac{Cn\zeta\eps^{\frac{m+1}{2}}}{\l\|f\|_{\Lp{\infty}(\M)}+\TV(f)\r^2}; \]
\add{indeed, from the first term on the right hand side in the expression for $F_{t,\veps, n}(f)$ we get a term $\sqrt{\zeta}$ which can be bounded from below by $\zeta (\l\|f\|_{\Lp{\infty}(\M)}+\TV(f)\r)^{-1}$; from the third term, on the other hand, we get a term $\zeta^{2/3}$ which is bounded from below by $\zeta \l\|f\|_{\Lp{\infty}(\M)}+\TV(f)\r^{-2/3}$, and we use the smallness of $\veps$ to conclude that $\veps^{(m+2)/3} \gtrsim \veps^{(m+1)/2}$. }

Hence, choosing $t=\frac{n\zeta}{2}$ in~\eqref{eq:U1} and $t=\frac{n(n-1)\zeta}{\l \|f\|_{\Lp{\infty}}+\TV(f)\r}$ in~\eqref{eq:U2} implies \add{
\begin{align*}
\Prob\l \la\frac{n}{n-1}\GTV_{n,\eps}(f) - \widetilde{\TV}_\eps(f)\ra \geq (1+\frac{1}{R(f)})\zeta\r & = \Prob\l |2U_1+U_2|\geq (1+\frac{1}{R(f)}) \zeta \r \\
 & \leq \Prob\l |U_1|\geq \frac{\zeta}{2} \text{ or } |U_2|\geq \frac{\zeta}{R(f)}\r \\
 & \leq \Prob\l|U_1|\geq\frac{\zeta}{2}\r + \Prob\l|U_2|\geq\frac{\zeta}{R(f)}\r \\
 & \leq C\exp\l-\frac{Cn\zeta^2\eps}{\l\|f\|_{\Lp{\infty}(\M)}+\TV(f)\r^2}\r \\
 & \hspace{1cm} + C\exp\l-\frac{Cn\zeta\eps^{\frac{m+1}{2}}}{\l\|f\|_{\Lp{\infty}(\M)}+\TV(f)\r^2}\r.
\end{align*}
}
%
Using Equation \eqref{eqn:TVEuclideanGeodesic} then implies that with probability at least $1-C \exp\left(-\frac{Cn\zeta \min\{\e^{\frac{m+1}{2}},\eps\zeta\}}{(\|f\|_{\Lp{\infty}(\M)} + \TV(f))^2}\right)$ we have that \add{
\[ \GTV_{n,\e}(f) \leq \sigma_\eta(1+C\e^2) \TV(f) + (1+\frac{1}{R(f)})\zeta. \qedhere \]
}
%
%
%
%
%
%
%
%
%
%
%
%
\end{proof}

%
%
%
%
%
%
%
%

\begin{remark}
It is possible to get quantitative estimates for the bias of the estimator $\GTV_{n,\e}(\remmath{u}\addmath{f})$ under smoothness assumptions on the function $f$.
In~\cite{arias2012normalized} a similar estimate has been carried out for functions of the form $f= \one_A$ where $A$ is a $\BV$ set, provided that the boundary of $A$ is smooth (see, e.g.,~\cite[Lemma 6]{arias2012normalized}).
An improved version of these bias estimates appears in~\cite{trillos2017estimating} (at least in the flat Euclidean setting).
For smooth functions $f$ a simple Taylor expansion allows us to obtain error estimates for the difference $| \TV(f) - \E(\GTV_{n,\e}(f)))|$.
However, for a general $\BV$ function for which we may not have information on its regularity, these estimates can not be used.
Fortunately, thanks to Proposition \ref{prop:biasineq} we do not need an estimate on the bias of the estimator.
\end{remark}


\add{
\begin{remark}
\label{rem:LocalGeomConstants}
It is worth highlighting that the constants $C$ in the statement of Proposition \ref{prop:VarianceEstimate} depend only on the dimension of the manifold $\M$, on the reach of $\M$ \addii{(where the reach is the largest number $R$ such that any point closer than $R$ to the manifold has a unique nearest point on the manifold)} to measure the discrepancy between Euclidean and geodesic distances in $\M$, on an upper bound on the absolute value of sectional curvatures of $\M$, and on the injectivity radius of $\M$. The aforementioned geometric quantities determine how small $\veps$ must be for all the comparisons between non-local and local total variation energies discussed in section \ref{sec:nonlocal} to hold.  
\end{remark}
}

\begin{proof}[Proof \add{of} upper bound in Theorem \ref{thm:cheeger-constant-conv}]
For simpler notation in this proof, for $A \subset \M$ we write $\Bal(A) = \min \{ \vol_\M(A),1 - \vol_\M(A)\}$ and $\Bal_n(A) = \min\{\nu_n(A_n), \nu_n(A^c)\}$.
For any fixed measurable set $A$ one can use Hoeffding's inequality to show that, with probability at least $1-2\exp(-2t^2 \addmath{(\Bal(A))^2} n)$, \add{$|\Bal(A) - \Bal_n(A)| \leq t \Bal(A)$}.

We now let $A=A^*$ be a solution to the continuum Cheeger problem and set \add{$f= \mathds{1}_{A^*}$ to be plugged in Proposition \ref{prop:VarianceEstimate};} existence of solutions to the Cheeger problem on $\M$ can be guaranteed using standard tools in the calculus of variations. Moreover, as the manifold $\M$ is connected, compact, and smooth we necessarily have $\Bal(A^*) > 0$.
By the definition of $\mathcal{C}_{n,\veps}$, and Proposition~\ref{prop:VarianceEstimate}, with probability at least \add{$1-2\exp(-2t^2 (\Bal(A^*))^2 n)-C \exp\left(-\frac{n}{R(f)}\zeta \min\{\e^{\frac{m+1}{2}},\eps\zeta\}\right)$ we have}
\begin{align}
\begin{split}
\mathcal{C}_{n,\veps} & \leq  \frac{\GTV_{n,\e}(\one_{A^*})}{\Bal_n(A^*)} \\
 & \leq \frac{(1 + C\e^2) \sigma_\eta \TV(\mathds{1}_{A^*}) + \addmath{ (1+ \frac{1}{R(f)})} \zeta }{\Bal_n(A^*)} \\
 & \leq  (1+2t) \left(\frac{(1 + C\e^2) \sigma_\eta \TV(\mathds{1}_{A^*}) + \addmath{(1+ \frac{1}{R(f)})} \zeta }{\Bal(A^*)}\right)
 \\
 & \leq \frac{\sigma_\eta \TV(\mathds{1}_{A^*})}{\Bal(A^*)} + C\left(\e^2 +  \zeta +  t \right)   \\
 & = \sigma_\eta \CM + C\left(\e^2 + \zeta + t\right), \label{eqn:cheeger-upper}
\end{split}
\end{align}
provided $0<t<1/2$. We can then select $t=\eps^{\frac{m+1}{4}}\zeta^{\frac{1}{2}}$ and obtain the desired result. \add{We notice that the constant $C$ appearing in front of $(\veps^2 +\zeta +t$) depends on the geometric quantities mentioned in Remark \ref{rem:LocalGeomConstants} as well as on the total variation and balance term of a minimizer $A^*$ of the continuum problem (i.e. a quantity that depends on the manifold $\M$).} 

\end{proof}

\subsection{Lower Bound}
\label{sec:Interpolation}

In this section we quantify the relationship
\[ \mathcal{\sigma_\eta}\CM\lesssim \mathcal{C}_{n,\veps}. \]
For that purpose we introduce an interpolation operator $\mathcal{I}_a$ mapping discrete functions to functions on $\M$ while reducing the total variation energy. More concretely, we construct a map 
\[ \mathcal{I}_a : \Lp{1}(\nu_n) \rightarrow \Ck{\infty}(\M) \]
with the crucial property that for \textit{all} $u \in \Lp{1}(\nu_n)$ we have (roughly speaking)
\[\sigma_\eta \TV( \mathcal{I}_a u) \lesssim \GTV_{n,\e}(u) \]
and for which $u \approx \I_a(u)$. Here $\sigma_\eta$ is as defined in \eqref{def:sigmaeta}. These estimates hold with very high probability, as described in more detail in Proposition \ref{prop:InterpMap} below. The interpolation operator takes the form 
\[ \mathcal{I}_a : \Lp{1}(\nu_n) \rightarrow \Ck{\infty}(\M) \]
\[ \mathcal{I}_a u = \Lambda_a (u\circ T_n) , \]
where $T_n: \M \rightarrow \M_n$ is a suitable transport map as defined in the following proposition, and $\Lambda_a$ is the smoothing operator defined in \eqref{def:Lambda}. The proof of the next proposition can be found in~\cite[Proposition \rem{2.10} \add{2.11}]{CalderTrillos}.

\begin{proposition}
\add{Let $\M$ be a smooth, connected, orientable, compact, manifold of dimension $m$ embedded in $\R^d$.
Then, there exists constants (that may depend on $\M$) $\theta_0,\delta_0,C,c,c'>0$ such that if $\delta\leq \delta_0$ , $ c' \log(n)/n \leq \theta^2 \delta^m $ and $\theta_0\geq \theta>0$ we have, with probability at least $1-ne^{-cn\theta^2\delta^m}$, that}
\rem{Suppose that $\M$ satisfies the assumptions at the beginning of Section \ref{sec:problemsetup} and suppose that $\delta, \theta>0$ are small enough quantities with $\delta\geq n^{-\frac{1}{m}}$. Then, with probability greater than $1- n \exp( - C n \theta^2 {\delta}^m )$,}
there exists a probability measure $\widetilde \nu_n$ with density function $\widetilde \rho _n : \M \rightarrow \R$
\rem{such that} \add{and a transport map $T_n$ between $\tilde{\nu}_n$ and $\nu_n$ (written $T_{n \sharp}\tilde \nu_n = \nu_n$) such that}
\begin{equation} 
\sup_{x \in \M} d_\M(x, T_n(x)) \leq {\delta} , 
\label{def:OT}
\end{equation}
and such that
\begin{equation}
\lVert 1 - \widetilde{\rho}_n \rVert_{\Lp{\infty}(\M)} \leq C \left( \theta +  \delta \right).
\label{eqn:densitiesbound}
\end{equation}
\label{prop:AuxiliaryDensity}
\end{proposition}


\begin{proposition}
\label{prop:InterpMap}
Let $\Lambda_a$ be as in \eqref{def:Lambda} for $a>0$ small enough and for given $\theta,\delta$ define $\I_a : \Lp{2}(\nu_n) \to \Ckc{\infty}(\M)$ by
\begin{equation}
\I_a(u) := \Lambda_a (u \circ T_n)
\end{equation}
where $T_n$ is an optimal transport map as in Proposition {\ref{prop:AuxiliaryDensity}}. Then, with probability at least $1- n \exp\left( - C n \theta^2 {\delta}^{m} \right)$, for every $u: \remmath{X}\addmath{\M_n} \rightarrow \R$ we have
\begin{enumerate}
\item $\lVert\I_a(u)\|_{\Ck{k}(\M)} \leq Ca^{-k} \|u\|_{{\Lp{\infty}(\nu_n)}}$.
\item $\|u\circ T_n- \I_a(u)\|_{\Lp{1}(\M)} \leq Ca\GTV_{n,\e}(u)$.
\item $\sigma_\eta \TV(\mathcal{I}_a u ) \leq (1+ C(\veps^2 +a + \frac{{\delta}}{\veps} + \theta ) ) \GTV_{n,\veps}(u) + C\left(\frac{\veps}{a^2}+a \right)\lVert u \rVert_{\Lp{\infty}(\nu_n)}$. 
\end{enumerate}
\end{proposition}

\begin{proof}
i) is a direct consequence of Proposition \ref{lem:Smoothening}.
To show ii) and  iii) we let 
\[ h:= \veps - 2 {\delta} >0 \]
and $\widetilde{\nu}_n$ be as in Proposition~\ref{prop:AuxiliaryDensity}.
From the fact that $T_{n \sharp} \widetilde{\nu}_n=\nu_n$ we can write for an arbitrary $u: \M \rightarrow \R$ 
\[ \GTV_{n,\veps}(u) = \frac{1}{\veps} \int_{\M}\int_{\M} \eta_\eps\l|T_n(x)-T_n(y)|\r |u(T_n(x)) - u(T_n(y))| \widetilde{\rho}_n(y) \widetilde{\rho}_n(x)\,\dd\vol_\M(y)\,\dd\vol_{\M}(x). \]
Now we notice that if $x,y \in \M$ are such that $d_\M(x,y)\leq h$ then necessarily $|T_n(x) - T_n(y)| \leq \veps$ (because $|x- T_n(x)|$, $|x-y|$ are bounded by $\delta$, $d_\M(x,y)$ respectively).
This means that for any $x,y \in \M$ it holds
\[ \eta\left( \frac{d_{\M}(x,y)}{h} \right) \leq \eta\left( \frac{|T_n(x)-T_n(y)|}{\veps} \right), \]
and thus
\[ \GTV_{n,\veps}(u)\geq\frac{h^m}{\veps^{m+1}}\int_\M \int_\M \eta_h\l d_\M(x,y)\r|u \circ T_n(x) - u \circ T_n(y)|\widetilde{\rho}_n(x)\widetilde \rho_n(y)\,\dd\vol_\M(y)\,\dd\vol_\M(x). \]
We now use the second inequality in Proposition \ref{prop:AuxiliaryDensity} to conclude that
\[ \GTV_{n,\veps}(u)\geq \left(1- \frac{2{\delta}}{\veps} \right)^{m+1}\left( 1 - C(\theta+{\delta})\right) \TV_{h}(u \circ T_n). \]
Using Corollary \ref{cor:smoothening-summary} then proves ii).

Again using Corollary \ref{cor:smoothening-summary} and the fact that $h \leq \veps$ we conclude that 
\begin{align*}
\sigma_\eta \TV\left(\Lambda_a (u \circ T_n) \right) & \leq (1+ C(\veps^2 + a))\left(1- \frac{2{\delta}}{\veps} \right)^{-(m+1)} \left( 1 - C(\theta+{\delta})\right)^{-1}\GTV_{n, \veps}(u) \\
 & \hspace{1cm} + C\left( \frac{\veps}{a^2}+a \right)\lVert u \rVert_{\Lp{\infty}(\nu_n)}.
\end{align*}
This expression can be simplified using the relative smallness assumptions on $\veps,\theta,\delta$, and in particular we can write
\[ \sigma_\eta \TV(\Lambda_a (u \circ T_n)) \leq (1+ C(\veps^2 +a + \frac{{\delta}}{\veps} + \theta +{\delta} ) ) \GTV_{n,\veps}(u) + C\left(\frac{\veps}{a^2}+a \right)\lVert u \rVert_{\Lp{\infty}(\nu_n)}. \]
This is precisely iii).
\end{proof}

\begin{proof}[Proof \add{of the} lower bound \add{in} Theorem \ref{thm:cheeger-constant-conv}]
We are now ready to obtain a lower bound for $\mathcal{C}_{n,\veps}$.
To achieve this let $A_n^*$ be a solution to the graph Cheeger cut problem, so that in particular
\[ \mathcal{C}_{n,\veps}= \frac{\GTV_{n,\veps}(\one_{A_n^*}) }{\min\{ \nu_n(A_n^*), \nu_n(A_n^{*c}) \}} ,\] 
and let $\one_{A}:=\one_{A_n^*}\circ T_n$.

As a first step we obtain an a-priori lower bound on $\min \{\nu_n(A_n^*), \nu_n(A_n^{*c})\}$.
First of all notice that $\widetilde{\nu}_n(A)=\nu_n(A_n^*)$ as well as $\widetilde{\nu}_n(A^c)= \nu_n (A_n^{*c})$, where $\widetilde{\nu}_n$ is the measure in Proposition \ref{prop:AuxiliaryDensity}. This is simply because $T_n$ is a transport map between $\widetilde{\nu}_n$ and $\nu_n$. On the other hand, notice that by \eqref{eqn:densitiesbound} and the smallness assumptions on ${\delta}$ and $\theta$ it follows that
\[ \frac{1}{2}\min\{ \nu(A) , \nu(A^c)\}\leq \min\{ \widetilde{\nu}_n(A) , \widetilde{\nu}_n(A^c) \} \leq (1+C(\delta+\theta))\min\{ \nu(A) , \nu(A^c)\}.\]
Now, as in the first part of the proof of Proposition~\ref{prop:InterpMap} (with the choice $\delta=\frac34$) we have
\[ C\GTV_{n, \veps}(\one_{A_n^*}) \geq \TV_{\veps/2}(\one_A), \]
for a constant $C$ that only depends on dimension. 
By the upper bound estimates we know that $\mathcal{C}_{n,\veps/2} \leq \sigma_\eta\CM+ C=: \frac{C_0}{2}$ and so we may use Lemma \ref{Lemma:SmallnessMini} to conclude that
\[ \frac{\beta_0}{2} \leq \frac{1}{2} \min \{\nu(A), \nu(A^c)\} \leq \min \{ {\nu}_n(A_n^*) ,{\nu}_n(A_n^{*c})\}, \] 
for some fixed constant $\beta_0>0$.

Let us now observe that for all $c \in[0,1]$, by Proposition~\ref{prop:InterpMap}(ii)
\[ \left | \lVert \mathcal{I}_a \one_{A_n^*} -c \rVert_{\Lp{1}(\M)} - \lVert \one_{A} -c \rVert_{\Lp{1}(\M)} \right| \leq \lVert \mathcal{I}_a \one_{A_n^*} - \one_{A_n^*}\circ T_n \rVert_{\Lp{1}(\M)} \leq C a \GTV_{n,\veps}(\one_{A_n^*}),\]
for a fixed constant $C$. 
From this uniform estimate it follows that
\[\left | \min_{c \in [0,1]}\lVert \mathcal{I}_a \one_{A_n^*} -c \rVert_{\Lp{1}(\M)} - \min_{c\in [0,1]}\lVert \one_{A} -c \rVert_{\Lp{1}(\M)} \right| \leq \lVert \mathcal{I}_a \one_{A_n^*} - \one_{A_n^*} \rVert_{\Lp{1}(\M)} \leq C a \GTV_{n,\veps}(\one_{A_n^*}).\]
Using the above inequalities and iii) in Proposition \ref{prop:InterpMap} we conclude that
\begin{align*}
\frac{\sigma_\eta \TV(\mathcal{I}_a \one_{A_n^*})}{ \lVert \mathcal{I}_a \one_{A_n^*}- m_1(\mathcal{I}_a \one_{A_n^*})\rVert_{\Lp{1}(\M)}} & \leq \frac{\l 1+C\l\eps^2+a+\frac{\delta}{\eps}+\theta\r\r \GTV_{n,\eps}(\one_{A_n^*}) + C\l\frac{\eps}{a^2}+a\r}{\min_{c\in[0,1]} \|\one_A -c\|_{\Lp{1}(\M)}- Ca\GTV_{n,\eps}(\one_{A_n^*})} \\
 & = \frac{\l 1+C\l\eps^2+a+\frac{\delta}{\eps}+\theta\r\r \GTV_{n,\eps}(\one_{A_n^*}) + C\l\frac{\eps}{a^2}+a\r}{\min\{\nu(A),\nu(A^c)\}- Ca\GTV_{n,\eps}(\one_{A_n^*})} \\
 & \leq \frac{\l 1+C\l\eps^2+a+\frac{\delta}{\eps}+\theta\r\r \GTV_{n,\eps}(\one_{A_n^*}) + C\l\frac{\eps}{a^2}+a\r}{\l 1-C(\delta+\theta)\r\min\{\nu_n(A_n^*),\nu_n(A_n^{*c})\}- Ca\GTV_{n,\eps}(\one_{A_n^*})}
\end{align*}
We divide each of the terms in the numerator and denominator of the right hand side of the above expression by $\min \{\nu_n(A_n^*), \nu_n(A_n^{*c}) \}$ and use the a-priori lower bound on this term to obtain
\[ \frac{\sigma_\eta \TV(\mathcal{I}_a \one_{A_n^*})}{ \lVert \mathcal{I}_a \one_{A_n^*}- m_1(\mathcal{I}_a \one_{A_n^*})\rVert_{\Lp{1}(\M)}}\leq \frac{(1+ C( \veps^2 +a + \frac{{\delta}}{\veps} + \theta)) \mathcal{C}_{n,\veps} + C(\frac{\veps}{a^2} + a) }{1- Ca \mathcal{C}_{n,\veps} - C(\delta+\theta)}. \]
Then, we use the upper bound on $\mathcal{C}_{n,\veps}$ that we found in Section~\ref{sec:Upper} to conclude that:
\begin{equation}\label{eqn:interp-est}
\frac{\sigma_\eta \TV(\mathcal{I}_a \one_{A_n^*})}{ \lVert \mathcal{I}_a \one_{A_n^*}- m_1(\mathcal{I}_a \one_{A_n^*})\rVert_{\Lp{1}(\M)}}\leq \mathcal{C}_{n,\veps} + C \CM\l \veps^2 +a + \frac{{\delta}}{\veps} + \theta\r + C \left(\frac{\veps}{a^2} +a + \theta + \delta +\zeta \right).
\end{equation}
with probability at least $1-Cne^{-cn\theta^2\delta^m}-Ce^{-cn\zeta\min\{\eps^{\frac{m+1}{2}},\eps\zeta\}}$.
Thanks to Lemma \ref{lem:CheegerFunctions}, we know that the left hand side is bounded below by $\sigma_\eta\CM$. 
Choosing $a=\sqrt[3]{\eps}$ we have
\begin{equation} \label{eqn:interp-est-update}
\sigma_\eta\CM \leq \frac{\sigma_\eta \TV(\mathcal{I}_a \one_{A_n^*})}{ \lVert \mathcal{I}_a \one_{A_n^*}- m_1(\mathcal{I}_a \one_{A_n^*})\rVert_{\Lp{1}(\M)}} \leq \mathcal{C}_{n,\veps} + C(\CM+1)\l \sqrt[3]{\eps} +\frac{\delta}{\eps} + \theta + \zeta\r
\end{equation}
as required.
\end{proof}




\subsection{Proof of Theorem \ref{thm:rates-Cheeger-Cuts}} \label{sec:CheegerCutsProof}

Having established convergence rates for Cheeger constants in Theorem \ref{thm:cheeger-constant-conv}, we now turn to proving convergence estimates for the Cheeger cuts $A_n^*$ towards continuum Cheeger sets $A^*$.
The first task will be to construct subsets of $\M$ which have the correct volume and for which we can control the perimeter.
This requires two steps: first, constructing a subset $\tilde A$ of $\M$ with smooth boundary which approximates $\mathcal{I}_a \one_{A_n^*}$ (which is the continuum object on which we can control the TV norm), and second, constructing a set $\hat A$ which adjusts the volume of $\tilde A$. These two steps are necessary because the stability results that are available in this context, that is the context of Proposition \ref{prop:global-perimeter-minimizers}, are concerned with mass-constrained perimeter minimizers.
The next lemma addresses this second step.

\begin{lemma}\label{lem:fix-mass}
Suppose that $\M$ satisfies Assumptions~\ref{assumptions}. \addii{Then for any $B>0$ and $\delta \in (0,1)$ there exists a $C_{\delta,B}>0$ so that for any $\tilde A \subset \M$ with $\Per(\tilde A) \leq B$ and $\delta \leq \vol_\M(\tilde A) \leq 1-\delta$, and any $\vv$ satisfying $\delta < \vol_\M(\tilde A) + \vv < 1-\delta$, there exists a set $\hat A$ satisfying}
\addii{
\begin{align*}
\vol_\M(\hat A) &= \vol_\M(\tilde A) + \vv,  \qquad \vol_\M(\hat A \Delta \tilde A) = |\vv|, \\
\Per(\hat A) &\leq \Per(\tilde A) + C_{\delta, B} |\vv| ^{\frac{m-1}{m}},
\end{align*}
where the constant $C_{\delta, B}$ only depends upon the values of $B$ and $\delta$ (and not, for example, on $\tilde{A}$).}
\end{lemma}

\begin{proof}
\addii{We restrict our attention to $\vv>0$, as the other case is analogous by considering $\tilde A^c$. We also notice that, by considering $\tilde A \cup B_\M(x,r)$ for appropriate $r$ that the conclusion is immediate for $\vv$ bounded away from zero. Hence we only need to focus our attention on $\vv$ in a neighborhood of zero.}

\addii{To begin, we claim that for any choice of $\delta,B$ and for any $\kappa \in (0,1)$ there exists an $R_0$ so that for any set $E$ satisfying $\Per(E) \leq B$ and $\delta \leq \vol_\M(E) \leq 1-\delta$ there exists a point $x$ such that $\vol_\M(E \cap B_\M(x,r)) \leq \kappa \vol_\M(B_\M(x,r))$ for all $r \leq R_0$.
Suppose that the claim were false. Then we would have a sequence of sets $E_n$, which satisfy the volume and perimeter bounds, $\delta \leq \vol_\M(E_n) \leq 1-\delta$ and  $\Per(E_n)\leq B$, and which satisfy
\begin{equation}\label{eqn:contradict-density}\liminf_n \inf_{x \in \M} \sup_{r \leq 1/n} \frac{\vol_\M(E_n \cap B_\M(x,r))}{\vol_\M(B_\M(x,r))}\geq \kappa.
\end{equation}
By BV compactness, we may assume that the $E_n$ converge to some set $E^*$ in $\Lp{1}(\M)$, and that $E^*$ satisfies the same volume and perimeter bounds.
By the Lebesgue differentiation theorem
\[ \frac{\vol_\M(E^*\cap B_\M(x,r))}{\vol_\M(B_\M(x,r))} \to \mathds{1}_{E^*}(x) \qquad \text{as } r\to 0 \]
for almost every $x\in\M$.
By Egorov's theorem we can find a subset $\M^\prime\subseteq \M$ such that the convergence is uniform over $\M^\prime$ and where $\vol_\M(\M^\prime)\geq 1-\frac{\delta}{2}$.
Hence, there exists $\eta>0$ such that
\[ \sup_{r<\eta} \frac{\vol_\M(E^*\cap B_\M(x,r))}{\vol_\M(B_\M(x,r))} < \frac{\kappa}{3} \]
for all $x\in\M^\prime\setminus E^*$.
Note that $\vol_\M(\M^\prime\setminus E^*) \geq \frac{\delta}{2}$ and therefore
\begin{equation}\label{eqn:lebesgue-diff}
\vol_\M\left(\left\{x : \sup_{r < \eta}\frac{\vol_\M(E^* \cap B_\M(x,r))}{\vol_\M(B_\M(x,r))}  < \frac{\kappa}{3} \right\}\right) > \frac{\delta}{2}.
\end{equation}
Now, by the weak maximal function inequality, we have that, for all $\alpha > 0$
\[
\vol_\M \left(\left\{ x: \sup_{r < \eta} \frac{1}{\vol_\M(B_\M(x,r))} \int_{B_\M(x,r)} |\mathds{1}_{E^*} - \mathds{1}_{E_n}| \, \dd\vol_\M   >\alpha \right\} \right)  \leq \frac{C_{\mathrm{max}}}{\alpha} \|\mathds{1}_{E^*}-\mathds{1}_{E_n}\|_{\Lp{1}};
\]
notice that the weak maximal function inequality is a direct consequence of Vitali covering lemma, which holds for arbitrary separable metric spaces; the constant $C_{\mathrm{max}}$ can be written in terms of $m$ (the intrinsic dimension of the manifold) and the constant $C_2$ appearing in~\eqref{eqn:diffgeoEsimate} below. By picking $\alpha = \frac{\kappa}{3}$ and then picking $n$ large enough that $\|\mathds{1}_{E^*}-\mathds{1}_{E_n}\|_{\Lp{1}} < \frac{\delta \kappa}{12 C_{\mathrm{max}}}$, we then have that 
\[
\vol_\M \left(\left\{ x: \sup_{r < \eta} \frac{1}{\vol_\M(B_\M(x,r))} \int_{B_\M(x,r)} |\mathds{1}_{E^*} - \mathds{1}_{E_n}| \, \dd\vol_\M   \leq \frac{\kappa}{3} \right\}\right)  \geq 1-\frac{\delta}{4}.
\]
This inequality, combined with the triangle inequality and Equation \eqref{eqn:lebesgue-diff} then necessarily contradicts equation \eqref{eqn:contradict-density}, which proves our claim.}

\addii{Now, let $\kappa = 1/2$ and apply the claim to conclude that for the set $\tilde A$ we can find a point $x$ so that for all $r < R_0$ we have that $\vol_\M(\tilde A\cap B_\M(x,r)) \leq 1/2 \vol_\M(B_\M(x,r))$. Using a classical volume estimate from Riemannian geometry (see e.g. \cite{gray1974volume}), we may deduce that for all $x\in \M$ and all small enough $r$ we have:
\begin{equation}
\label{eqn:diffgeoEsimate}
C_1 r^m \leq \vol_\M(  B_\M(x,r)) \leq C_2r^m.
\end{equation}
We then select $r^*$ so that $\vol_\M( B_\M(x,r^*) \setminus \tilde A) = \vv$, and define $\hat A := \tilde A \cup B_\M(x,r^*)$. We note that by the choice of the point $x$, we have that
\[ \vartheta = \vol_\M(B_\M(x,r^*)\setminus\tilde{A}) = \vol_\M(B_\M(x,r^*)) - \vol_\M(B_\M(x,r^*)\cap \tilde{A}) \geq \frac12 \vol_\M(B_\M(x,r^*)) \geq \frac{C_1 r^{*m}}{2}, \]
and clearly
\[ \vartheta = \vol_\M(B_\M(x,r^*)\setminus\tilde{A}) \leq \vol_\M(B_\M(x,r^*)) \leq C_2 r^{*m}. \]
So,
\[
\tilde{C}_1 \vv \leq r^{*m} \leq \tilde{C}_2 \vv,
\]
with constants that depend only upon $B,\delta$, and not upon the particular set $\tilde A$.
Moreover, since for all $x \in \M$ and all small enough $r$ we have
\begin{equation} \label{eq:PerBallBound}
\Per(B_\M(x,r)) \leq Cr^{m-1},
\end{equation}
by~\cite{morgan1998riemannian} or~\cite{morgan2000some}, then the perimeter of $\hat{A}$ can be bounded by
\begin{align*}
\Per(\hat{A}) & \leq \Per(\widetilde{A}) + \Per(B_{\M}(x,r^*)) \\& \leq \Per(\widetilde{A}) + C(r^*)^{m-1} 
 \\& \leq \Per(\widetilde{A}) +  C\vv^{\frac{m-1}{m}}. \qedhere
\end{align*}
}
\end{proof}

\begin{proof}[Proof of Theorem \ref{thm:rates-Cheeger-Cuts}]
We use \add{the notation $\Bal(A)$} as introduced in the proof of the upper bound in Theorem~\ref{thm:cheeger-constant-conv}.

First notice that for a given discrete minimizer $A_n^*$ we can build a set $\tilde A \subseteq \M$  with the following properties (for $0 < z < 1/2$): 
\begin{enumerate}
	\item $\Per(\tilde A) \leq \frac{\TV(\mathcal{I}_a \one_{A_n^*})}{(1-2z)}$.
	\item $\lVert \one_{A_n^*} \circ T_n - \mathds{1}_{\tilde A}\rVert_{\Lp{1}(\M)} \leq \frac{Ca}{z}$.
\end{enumerate}
To see this, notice that by the coarea formula, and the fact that $0 \leq \mathcal{I}_a \one_{A_n^*} \leq 1$, we have that
\[ \TV(\mathcal{I}_a \one_{A_n^*}) = \int_0^1 \Per(\{\mathcal{I}_a \one_{A_n^*} \leq s\}) \,\dd s \geq \int_z^{1-z} \Per(\{\mathcal{I}_a \one_{A_n^*\leq s}\}) \, \dd s. \]
In turn there must exist a $ \widetilde{z} \in (z, 1-z)$ so that the set $\tilde A = \{\mathcal{I}_a \one_{A_n^*} > \widetilde{z} \}$ satisfies $\Per(\tilde A) \leq \frac{\TV(\mathcal{I}_a \one_{A_n^*})}{(1-2z)}$.
To obtain the second inequality we use Proposition~\ref{prop:InterpMap}, $\mathcal{I}_a\one_{A_n^*}\geq \tilde{z}\one_{\tilde{A}}$ and $1-\mathcal{I}_a\one_{A_n^*}\remmath{(x)} \geq (1-\tilde{z})\one_{\widetilde{A}^c}$, to estimate
\begin{align*}
Ca \GTV_{n,\e}(\one_{A_n^*}) & \geq \| \one_{A_n^*} \circ T_n - \mathcal{I}_a \one_{A_n^*}\|_{\Lp{1}(\M)} \\
 & = \int_{\M} \la \one_{A_n^*}(T_n(x)) - \mathcal{I}_a\one_{A_n^*}(x)\ra \, \dd \vol_{\M}(x) \\
 & = \int_{T_n^{-1}(A_n^*)} \l 1-\mathcal{I}_a\one_{A_n^*}(x)\r \, \dd \vol_{\M}(x) + \int_{\M\setminus T_n^{-1}(A_n^*)} \mathcal{I}_a\one_{A_n^*}(x) \, \dd \vol_{\M}(x) \\
 & \geq (1-\tilde{z}) \int_{T_n^{-1}(A_n^*)} \one_{\widetilde{A}^c}(x) \, \dd \vol_{\M}(x)+ \tilde{z} \int_{\M\setminus T_n^{-1}(A_n^*)} \one_{\widetilde{A}}(x) \, \dd \vol_{\M}(x) \\
 & \geq \min(\widetilde{z},1-\widetilde{z}) \lVert\one_{A_n^*} \circ T_n - \mathds{1}_{\tilde A}\rVert_{\Lp{1}(\M)} \\
 & \geq z \lVert\one_{A_n^*} \circ T_n - \mathds{1}_{\tilde A}\rVert_{\Lp{1}(\M)},
\end{align*}
and notice that the upper bound for $\mathcal{C}_{n,\veps}$ can be used to bound the left hand side of the above expression by a constant times $a$.

\add{We will use} \rem{From} the bound\add{s}
\begin{align*}
& \lda \mathcal{I}_a\one_{A_n^*} - m_1(\mathcal{I}_a\one_{A_n^*})\rda_{\Lp{1}(\M)} - \Bal(\tilde{A}) \\
& \hspace{0.9cm} = \max\lb \min_{c\in [0,1]} \lda \mathcal{I}_a\one_{A_n^*} - c\rda_{\Lp{1}(\M)} - \underbrace{\vol_{\M}(\widetilde{A})}_{=\|\one_{\widetilde{A}}\|_{\Lp{1}(\M)}}, \min_{d\in [0,1]} \lda \mathcal{I}_a\one_{A_n^*} - d\rda_{\Lp{1}(\M)} - \underbrace{\vol_{\M}(\widetilde{A}^c)}_{=\|1-\one_{\widetilde{A}}\|_{\Lp{1}(\M)}} \rb  \\
& \hspace{0.9cm} \leq \lda \mathcal{I}_a\one_{A_n^*} - \one_{\tilde{A}} \rda_{\Lp{1}(\M)} \\
& \hspace{0.9cm} \leq \lda \mathcal{I}_a\one_{A_n^*} - \one_{A_n^*}\circ T_n \rda_{\Lp{1}(\M)} + \lda \one_{A_n^*}\circ T_n - \one_{\tilde{A}} \rda_{\Lp{1}(\M)} \\
& \hspace{0.9cm} \leq \frac{Ca}{z}
\end{align*}
and
\begin{align*}
\addmath{\Bal(\tilde{A})} & \addmath{= \min \left\{ \vol_{\M}(\tilde{A}),\vol_{\M}(\tilde{A}^c)\right\} }\\
 & \addmath{= \min\lb \| \one_{\tilde{A}}\|_{\Lp{1}(\M)}, \|1-\one_{\tilde{A}}\|_{\Lp{1}(\M)} \rb }\\
 & \addmath{\geq \min \Bigg\{\|\one_{A_n^*}\circ T_n\|_{\Lp{1}(\M)} - \underbrace{\|\one_{\tilde{A}}-\one_{A_n^*}\circ T_n\|_{\Lp{1}(\M)}}_{\leq \frac{Ca}{z}\GTV_{n,\eps}(\one_{A_n^*})\leq \frac{Ca}{z}}, }\\
 & \qquad\qquad\qquad\qquad \addmath{\|1-\one_{A_n^*}\circ T_n\|_{\Lp{1}(\M)} - \|\one_{\tilde{A}}-\one_{A_n^*}\circ T_n\|_{\Lp{1}(\M)} \Bigg\} }\\
 & \addmath{\geq \min\lb \|\one_{A_n^*}\circ T_n\|_{\Lp{1}(\M)}, \|1-\one_{A_n^*}\circ T_n\|_{\Lp{1}(\M)} \rb - \frac{Ca}{z} }\\
 & \addmath{= \min\lb \int_{\M} \one_{A_n^*}(T_n(x)) \, \dd \vol_{\M}(x), \int_{\M} \one_{(A_n^*)^c}(T_n(x)) \, \dd \vol_{\M}(x) \rb - \frac{Ca}{z} }\\
 & \addmath{\geq c\min\lb \int_{\M} \one_{A_n^*}(T_n(x)) \tilde{\rho}_n(x) \, \dd \vol_{\M}(x), \int_{\M} \one_{(A_n^*)^c}(T_n(x)) \tilde{\rho}_n(x) \, \dd \vol_{\M}(x) \rb - \frac{Ca}{z} }\\
 & \addmath{= c\min\lb \nu_n(A_n^*), \nu_n((A_n^*)^c) \rb - \frac{Ca}{z}. }
\end{align*}
Assuming that $\frac{a}{z}$ is sufficiently small (which by later choices is equivalent to $\eps$ being sufficiently small), and the a-priori lower bound on $\nu_n(A_n^*)$ derived in the proof of the lower bound in Theorem~\ref{thm:cheeger-constant-conv} we can assume that $\Bal(\tilde{A})\geq c$ for some constant $c>0$ independent of all other parameters.

The perimeter estimate $\Per(\tilde A) \leq \frac{\TV(\mathcal{I}_a \one_{A_n^*})}{(1-2z)}$ then implies, if we further restrict to $z<\frac14$, that
\begin{align*}
& \frac{\sigma_\eta \TV(\one_{\tilde A})}{\Bal(\tilde A)} - \frac{\sigma_\eta\TV(\mathcal{I}_a(\one_{A_n^*}))}{ \lVert \mathcal{I}_a \one_{A_n^*}- m_1(\mathcal{I}_a \one_{A_n^*})\rVert_{\Lp{1}(\M)}} \\
  &\leq \underbrace{\frac{\sigma_\eta\TV(\mathcal{I}_a(\one_{A_n^*}))}{ \lVert \mathcal{I}_a \one_{A_n^*}- m_1(\mathcal{I}_a \one_{A_n^*})\rVert_{\Lp{1}(\M)}}}_{\addmath{\leq C \text{ by \eqref{eqn:interp-est-update}}}} \left( \frac{(1 + 4z)\lVert \Lambda_a \one_A- m_1(\Lambda_a\one_{A})\rVert_{\Lp{1}(\M)} - \Bal(\tilde A)}{\Bal(\tilde A)} \right)\\
  &\leq C\left( \frac{a}{z} + z \right)
\end{align*}
where $A$ is the set satisfying $\one_{A_n^*}\circ T_n = \one_A$.
Using equations \eqref{eqn:cheeger-upper} and \eqref{eqn:interp-est-update}, and choosing $a=\sqrt[3]{\eps}$, gives
\begin{align*}
  0 &\leq \frac{\sigma_\eta \TV(\one_{\tilde A})}{\Bal(\tilde A)} - \mathcal{C}_{n,\e} + \mathcal{C}_{n,\e} - \sigma_\eta\CM \\
  &\leq \frac{\sigma_\eta \TV(\one_{\tilde A})}{\Bal(\tilde A)} - \frac{\sigma_\eta\TV(\mathcal{I}_a(\one_{A_n^*}))}{ \lVert \mathcal{I}_a \one_{A_n^*}- m_1(\mathcal{I}_a \one_{A_n^*})\rVert_{L^1(\M)}} + C\left(\sqrt[3]{\eps} + \frac{\delta}{\e} + \theta  + \zeta \right) \\
  &\leq C\left(\sqrt[3]{\eps} + \frac{\delta}{\e} + \theta + \frac{\sqrt[3]{\eps}}{z} + z + \zeta \right) =: C \kappa
\end{align*}

By Assumption~\ref{assumption:Cheeger-sets} and Proposition~\ref{prop:global-perimeter-minimizers} there exists a continuum Cheeger $A^*$ such that $|\vol_\M(\tilde A) - \vol_\M(A^*)| \leq C \kappa^{1/2}$. \addii{ Moreover, there is a $\delta_\M$ (independent of $A^*$) such that $\delta_\M \leq \vol_\M(A^*) \leq 1-\delta_\M$, as it follows from well-known asymptotics~\cite{morgan2000some} for the isoperimetric profile near zero and compactness of the perimeter functional, which imply that solutions to the continuum Cheeger problem can not have arbitrarily small volume; recall the discussion right after Remark \ref{rem:OtherConv1}.}

\addii{ 
Invoking Lemma \ref{lem:fix-mass} with $\delta= \delta_\M/2$ and $B = 2 \sigma_\eta \mathcal{C}_\M $, and assuming that $\kappa$ is small enough (which can be guaranteed after setting $z:= \veps^{1/6}$ and setting all the other parameters to be small enough),} we can then find a set $\hat A$ so that $\vol_\M(\hat A) = \vol_\M(A^*)$,  $\Per(\hat A) - \Per(A^*) \leq C \kappa^{\frac{m-1}{2m}}$ 
and $\vol_\M(\hat A \Delta \tilde A) \leq C\kappa^{1/2}$.
%
Via Proposition~\ref{prop:global-perimeter-minimizers} we immediately obtain
\begin{displaymath}
\alpha(\hat A) \leq C\kappa^{\frac{m-1}{4m}},
\end{displaymath}
as desired.
Combining these estimates, 
\begin{align*}
\lVert \mathds{1}_{A^*} - \one_{A_n^*} \circ T_n \rVert_{\Lp{1}(\M)} & \leq  \lVert \one_{A_n^*} \circ T_n - \mathds{1}_{\tilde A} \rVert_{\Lp{1}(\M)} + \vol_\M(\tilde A ~\Delta~ \hat A) + \alpha(\hat A) \\
 & \leq C\left( \sqrt[6]{\eps} + \kappa^{1/2} + \kappa^{\frac{m-1}{4m}}\right) \\
 & \leq C \kappa^{\frac{m-1}{4m}}.
\end{align*}
This concludes the proof.
\end{proof}

\begin{remark}
\label{rem:OtherConv2}
To show the statement in Remark \ref{rem:OtherConv1} first notice that 
since $T_{n \sharp} \widetilde{\nu}_n =\nu_n $ then
\[  \nu_n\left( A^*  \Delta A_n^*   \right) = \lVert \mathds{1}_{A_n^*}\circ T_n - \mathds{1}_{A^*}\circ T_n  \rVert_{\Lp{1}(\widetilde{\nu}_n)} \]
and in turn by the triangle inequality
\[ \lVert \mathds{1}_{A_n^*}\circ T_n - \mathds{1}_{A^*}\circ T_n  \rVert_{\Lp{1}(\widetilde{\nu}_n)} \leq    \lVert \mathds{1}_{A_n^*}\circ T_n - \mathds{1}_{A^*}  \rVert_{\Lp{1}(\widetilde{\nu}_n)} + \lVert \mathds{1}_{A^*} - \mathds{1}_{A^*}\circ T_n  \rVert_{\Lp{1}(\widetilde{\nu}_n)}\] 	
The first term is the bound from Theorem \ref{thm:rates-Cheeger-Cuts} (enlarging the constant to account for the change of measure using Proposition \ref{prop:AuxiliaryDensity}). On the other hand the second term can be estimated by the volume of the tubular neighborhood of width $2\delta$ around $\partial A^*$ (where $\delta$ appears in \eqref{def:OT}) and thus, given the smoothness of $\partial A^*$ by $C \Per(A^*) \delta$.
\end{remark}

\begin{remark}\label{rem:not-sharp}
In the proof of Theorem \ref{thm:rates-Cheeger-Cuts} we have relied upon the tools for mass-constrained perimeter minimizers, and fixed the mass separately.
This has caused various losses in the estimates.
For example, we have lost a power of two twice, once when comparing the mass and once when comparing the mass-constrained minimizers.
The $\frac{m-1}{m}$ is also an artifact of needing to fix the mass in order to use the isoperimetric stability results.
It seems very likely that the $\kappa^{\frac{m-1}{4m}}$ should in reality be a $\kappa^{1/2}$.
Similarly, relying on stability results for sets, as opposed to the relaxed formulation in Lemma \ref{lem:CheegerFunctions} necessitates the first step in the proof of Theorem \ref{thm:rates-Cheeger-Cuts}, which introduces the parameter $z=\sqrt[6]{\eps}$ into $\kappa$, and accordingly decreases the power on $\e$ in $\kappa$.
Pursuing such issues, by obtaining improved stability results for total variation minimization problems, is outside the scope of this work.
\end{remark}

\appendix
\section{}

\begin{proof}[Distance estimates: Proof of \eqref{app:DistanceEstimate}]
In the context of $\eqref{app:DistanceEstimate}$, the quantity $|w|_x$ is the same as $d_\M(x,y)$ where $y =\exp_{x}(w)$, and $|\tilde{w}|_{\tilde x}$ is the same as $d_\M(\tilde x , \exp_{y}(PT_{x,y}(v)) )$. This is known as a Levi-Civita parallelogramoid (see \cite{Cartan} Chapter X, section VI), and satisfies the estimate
\[
  d_\M^2(\tilde x , \exp_{y}(PT_{x,y}(v)) ) - d_\M^2(x,y) \approx \frac{8}{3}\langle R(v,w) v,w \rangle_x,
\]
where the ``$\approx$'' means that we have neglected higher order terms and $R$ is the Riemann curvature tensor.
Since $R$ is bilinear then we have the bound
\[ \la d_\M^2(\tilde x , \exp_{y}(PT_{x,y}(v)) ) - d_\M^2(x,y) \ra \leq C \lda R(v,w)v\rda_x \| w\|_x \leq C \|v\|_x^2 \| w\|_x^2 \leq Ca^2h^2 \]
which gives the desired estimate.
%
%
\end{proof}

\begin{proof}[Comparison of Jacobians between exponential maps: Proof of \eqref{app:JacobiansEstiamtes}.]
Fix $x,y \in \M$ for which $d_\M(x,y) \leq h $ and for convenience let $t_0:= d_\M(x,y)$. Let $v \in B(0,2a) \subseteq \R^m$, and write the coordinates of $v$ as $v_1, \dots, v_m$. Let $t\in [0, t_0] \rightarrow \gamma(t)$ be the unit speed geodesic which at time zero is at $x$ and at time $t_0$ is at $y$. Let $e_1, \dots, e_m$ be an orthonormal basis for $T_x \M$ and for each $i=1, \dots, m$ let 
\[ t\in [0, t_0] \mapsto E_i(t) \] 
be the parallel transport of $e_i$ along $\gamma$. Finally, define the map
\[ P:(t,v) \longmapsto \exp_{\gamma(t)}\left( \sum_{i=1}^m v_i E_i(t) \right) \] 
Notice that for any given $t$, the map $P(t, \cdot)$ defines a coordinate chart for the set $B_\M(\gamma(t),2a)$. In what follows for all $i,j =1, \dots, n$ we quantify how the inner product $\left \langle \frac{\partial}{\partial v_i} , \frac{\partial}{\partial v_j} \right \rangle_{P(t,v)} $ changes in $t$ for any given $v$. 
First of all notice that when $v=0$, we have for all $t$ 
\[ \left \langle \frac{\partial}{\partial v_i} , \frac{\partial}{\partial v_j} \right \rangle_{\gamma(t)}= \delta_{ij}. \]
This is because for all $t$, the coordinate chart $P(t,\cdot)$ is a chart of normal coordinates at $\gamma(t)$.

Let us now fix an arbitrary $v \not =0$ in the ball $B(0,2a)$ and define
\[ \varphi(s,t):= \left \langle \frac{\partial}{\partial v_i} , \frac{\partial}{\partial v_j} \right \rangle_{P\left(t,s\frac{v}{|v|}\right)}, \quad s\in [0,|v|], \quad t \in [0,t_0]. \]
Notice that the function $\varphi$ is smooth. Also, just as before,
\[ \varphi(0,t)=\delta_{ij}, \quad \forall t \in [0,t_0]. \]
In particular, it follows that 
\begin{align*}
|\varphi(|v|,t_0) - \varphi(|v|,0) | & = | \varphi(|v|,t_0) - \varphi(0,t_0) +\varphi(0,0) - \varphi(|v|,0) | 
\\&= \left | \int_{0}^{|v|} \frac{\partial }{\partial s} \varphi(s,t_0) \, \dd s - \int_{0}^{|v|} \frac{\partial }{\partial s} \varphi(s,0) \, \dd s \right |  
\\ & \leq \int_{0}^{|v|} \left | \frac{\partial }{\partial s} \varphi(s,t_0) - \frac{\partial }{\partial s} \varphi(s,0) \right | \, \dd s.
\end{align*}
On the other hand, for any given $s \in [0, |v|]$ we have
\[ \frac{\partial }{\partial s} \varphi(s,t_0) - \frac{\partial }{\partial s} \varphi(s,0) = \int_{0}^{t_0} \frac{\partial^2 \varphi(s,t) }{\partial t \partial s} \,\dd t. \]
The smoothness of the function $\varphi$ can then be used to conclude that for all $s \in [0, |v|]$ and all $t \in [0, t_0]$ we have 
\[ \left \lvert \frac{\partial^2 \varphi(s,t) }{\partial t \partial s} \right \rvert \leq C \]
for some constant $C$.
In fact, using the compactness of $\M$, the constant $C$ can be chosen to be independent of $x \in \M$, $y \in B_\M(x,{ h})$ and $v \in B(0,2a)$. Putting all the above estimates together we deduce
\[ \lvert \varphi(|v|, t_0) - \varphi(|v|,0) \rvert \leq C t_0 |v|. \]
Hence, for all $i,j$,
\[ \left \lvert \left \langle \frac{\partial }{\partial v_i} , \frac{\partial }{\partial v_j}\right \rangle_{P(0,v)} - \left \langle \frac{\partial }{\partial v_i} , \frac{\partial }{\partial v_j}\right \rangle_{P(t_0,v)} \right \rvert \leq C ad_\M(x,y) . \]
Since the Jacobian $J_x(v)$ is given by the square root of the determinant of the Gramian matrix $ \left[\left \langle \frac{\partial }{\partial v_i} , \frac{\partial }{\partial v_j}\right \rangle_x \right]_{i,j} $, and as $J_x(v)$ is bounded away from zero, we may estimate
\[ |J_x(v) - J_y(PT_{x,y}(v))| \leq C a d_\M(x,y) \leq Ca h , \]
for all $v$ with Euclidean norm less than $a$.
\end{proof}

Following~\cite{docarmo1992riemannian} let us define the metric on $\T\M$ as follows.
Given $\gamma_1, \gamma_2:(-\beta, \beta) \rightarrow \T \M$ two smooth integral curves $\gamma_1: t \mapsto (x_1(t),v_1(t))$ and $\gamma_2:t\mapsto (x_2(t), v_2(t)) $ associated to the vector fields $V_1,V_2$ (on $\T\M$), passing through $(x,v)$ at time zero (i.e.$\gamma_1(0)=(x,v), \gamma_2(0)=(x,v)$), we define
\begin{equation}
\label{def:RiemanTM}
\langle V_1 , V_2 \rangle_{x,v} := \langle \dd\pi(V_1), \dd\pi(V_2) \rangle_x + \left\langle \frac{\DD}{\dd t}v_1(0) , \frac{\DD}{\dd t}v_2(0) \right\rangle_x, 
\end{equation}
where here $\pi$ is the projection map $\pi: (\tilde x,\tilde v) \in \T\M \mapsto \tilde x \in \M$ and $\frac{\DD}{\dd t}v_1(0)$ and $\frac{\DD}{\dd t} v_2(0)$ represent the covariant derivatives at time zero of the vector fields $t \mapsto v_1(t)$ and $t \mapsto v_2(t)$ along the curves $t \mapsto x_1(t)$ and $t \mapsto x_2(t)$ respectively.
It is straightforward to check that this definition does not depend on the specific choice of curves. 

\begin{proof}[Estimation of Jacobian of $\widetilde \Psi$ in $\mathcal{B}_{h,a}$: Proof of \eqref{app:JacobiansPsi}]
 First of all, we start by making precise the metric for $\T^2\M$ that we referred to in the proof of Proposition \ref{prop:ConvolutionMonotone}. This metric is analogous to the one $\T \M$ was endowed with in \eqref{def:RiemanTM}. Given two vector fields $V_1, V_2$ on $\T^2 \M$, let $t \mapsto \alpha_1(t)=(x_1(t),v_1(t), \tilde v_1(t))$ and $t \mapsto \alpha_2(t)=(x_2(t),v_2(t), \tilde v_2(t))$ be two curves on $\T^2\M$ satisfying
\[ \alpha_1(0)=(x,v,\tilde v) = \alpha_2(0), \]
\[ \dot{\alpha}_1(0) = V_1(x,v,\tilde v) , \quad \dot{\alpha}_2(0)= V_2(x,v,\tilde v)\]
and define
\begin{equation}
g_{x,v,\tilde v}(V_1, V_2) := \left\langle\dd\pi(V_1) , \dd\pi(V_2) \right\rangle_x + \left\langle \frac{\DD}{\dd t}v_1(0) , \frac{\DD}{\dd t}v_2(0) \right\rangle_x + \left\langle \frac{\DD }{\dd t}\tilde v_1(0) , \frac{\DD}{\dd t}\tilde v_2(0) \right\rangle_x.
\label{def:RiemanT2M}
\end{equation}
Since $\pi:\T^2\M \to \M$ projects onto the first coordinate the first term on the right hand side of \eqref{def:RiemanT2M} can be seen as $\langle \dot x_1,\dot x_2 \rangle_x$. 



 For a fixed $x_0 \in \M$ we may consider the set of all $(x,v,v') \in \mathcal{B}^2_{2h,2a}$ for which $x \in B_\M(x_0, r)$ for some order one quantity $r$. 
 We consider a neighborhood $\mathcal{V}$ of the origin in $\R^m$ small enough to ensure that the map
\begin{equation} 
(p_1, \dots, p_m) \in \mathcal{V} \mapsto \exp_{x_0}\left(\sum_{i=1}^m p_i E_i \right) 
\label{aux:coordinatesM}
\end{equation}  
is a diffeomorphism, where $E_1, \dots, E_m$ is an arbitrary fixed orthonormal basis for $\T_{x_0}\M$. We then consider the coordinate chart
\begin{equation}
(p_1, \dots, p_m, y_1, \dots, y_m, z_1, \dots, z_m) \in \mathcal{V} \times \mathcal{Y} \times \mathcal{Z} \mapsto \left( \exp_{x}\left(\sum_{i=1}^m p_i E_i \right), \sum_{i=1}^m y_i \frac{\partial }{\partial p_i}, \sum_{i=1}^m z_i \frac{\partial }{\partial p_i} \right) 
\label{aux:CoordinatesT2M}
\end{equation}
Requiring $a,h,r$ to be small enough, allows us to assume that the set of points $(x,v,v') \in \mathcal{B}^2_{2h,2a}$ for which $x \in B_\M(x_0, r)$ is contained in the neighborhood $\mathcal{V} \times \mathcal{Y} \times \mathcal{Z}$. Moreover, since
\[ \widetilde{\Psi}(p,0,0)= (p,0,0), \]
by the smoothness and compactness of $\M$ we may assume that the image  under $\widetilde{\Psi}$ of the set of points $(x,v,v') \in \mathcal{B}^2_{2h,2a}$ for which $x \in B_\M(x_0, r)$ is still contained in $\mathcal{V} \times \mathcal{Y} \times \mathcal{Z}$.

In what follows, we may abuse notation slightly and use $\widetilde{\Psi}$ to denote the map $\widetilde{\Psi}$ in the above coordinates (restricting the domain to a smaller neighborhood than the one where we defined the coordinates, so that the image of $\Psi$ is  contained in the neighborhood where the coordinates were defined). We use $\mathcal{G}(p,y,z)$ to denote the Grammian matrix at the point $(p,y,z)$ in the above coordinates.

A straightforward computation using the parallel transport equations in coordinates (and in particular also the geodesic equations) shows that the Jacobian matrix $\frac{\partial \widetilde{\Psi}}{\partial(p,y,z)}(p,0,0)$ satisfies 
\[ \frac{\partial \widetilde{\Psi}}{\partial(p,y,z)}(p,0,0) =\Id \]
where $\Id$ is the $m \times m$ identity matrix. We can then use the smoothness and compactness of $\M$  to conclude that
\[  \det(\mathcal{G}(p,y,z) ) = \det(\mathcal{G}(p,0,0) ) + O(a)  \]
\[ \det(\mathcal{G}( \widetilde{\Psi}(p,y,z)) ) = \det(\mathcal{G}(\Psi(p,0,0)) ) + O(a) = \det(\mathcal{G}(p,0,0) ) + O(a)   \]
\[   \frac{\partial \widetilde{\Psi}}{\partial(p,y,z)}(p,y,z) = \Id + O(a)    \]
for all $y,z$ with $|y|\leq h, |z| \leq a$, $h \leq a$, with a uniform  $O(a)$ term that does not depend on $p$ nor the point $x_0 \in \M$ around which we constructed the normal coordinates. The bottom line is that 
\[ \sqrt{\det(\mathcal{G}(p,y,z) )} \left \lvert \frac{\partial \widetilde{\Psi}}{\partial(p,y,z)} \right \rvert^{-1} = (1+ O (a)) \cdot \sqrt{\det(\mathcal{G}(\widetilde{\Psi}(p,y,z)))},\]
for all $z$ with Euclidean norm less than $a$ and all $y$ with norm less than $h$. This is \eqref{app:JacobiansPsi} in coordinates.
\end{proof}

\bibliography{Trend-filter}
\bibliographystyle{siam}

\end{document}